\newif\ifdraftmode
\draftmodefalse
\documentclass[10pt]{article}

\ifdraftmode
\fi

\newcommand\thisShorttitle{Convergence bounds for local least squares approximation}
\newcommand\thisTitle{\thisShorttitle}
\newcommand\thisSubject{math.NA, stat.ML}
\newcommand\thisAuthor{Philipp Trunschke}
\newcommand\thisKeywords{least squares approximation, sample complexity, tensor networks, manifolds, positive reach}

\usepackage{arxiv}

\usepackage[utf8]{inputenc}      
\usepackage[T1]{fontenc}         
\usepackage[dvipsnames]{xcolor}  
\usepackage{hyperref}            
\usepackage{url}                 
\usepackage{doi}                 
\usepackage{microtype}           
\usepackage{graphicx}
\hypersetup{
    breaklinks,
    raiselinks=true,
    colorlinks,
    citecolor=gray,
    filecolor=gray,
    linkcolor=gray,
    urlcolor=gray,
    bookmarksopenlevel=1,    
    bookmarksopen=true,      
    bookmarksnumbered=true,  
    bookmarkstype=toc,       
    pdfdisplaydoctitle=true, 
    pdfstartview=FitH,       
    pdfpagemode=UseOutlines, 
    pdfpagelayout=OneColumn, 
    pdftitle={\thisTitle},
    pdfsubject={\thisSubject},
    pdfauthor={\thisAuthor},
    pdfkeywords={\thisKeywords},
}

\usepackage{enumitem}            
\usepackage[ruled,linesnumbered]{algorithm2e}  
\usepackage{caption}             
\usepackage{subcaption}          
\usepackage{float}               
\usepackage{placeins}[section]   

\usepackage{csquotes}  
\usepackage[
    bibstyle=alphabetic,
    citestyle=alphabetic,
    sorting=nyt, 
    sortcites=true,
    giveninits=true,
    maxbibnames=99,
    backend=biber,
    maxalphanames=5,
    backref=true,
    natbib=true,     
    url=false,
    date=year, 
    abbreviate=false,
]{biblatex}

\DeclareFieldFormat{eprint:hal}{%
    \mkbibacro{HAL}\addcolon\space
    \ifhyperref
        {\href{https://hal.archives-ouvertes.fr/#1}{%
        \nolinkurl{#1}
        \iffieldundef{eprintclass}
            {}
            {\addspace\UrlFont{\mkbibbrackets{\thefield{eprintclass}}}}}}
        {%
        \nolinkurl{#1}
        \iffieldundef{eprintclass}
            {}
            {\addspace\UrlFont{\mkbibbrackets{\thefield{eprintclass}}}}}}
\DeclareFieldAlias{eprint:HAL}{eprint:hal}

\DeclareSourcemap{
  \maps[datatype=bibtex]{
    \map{
      \step[fieldset=issn, null]
    }
  }
}
\renewbibmacro{in:}{}  

\def\keywordname{{\bfseries \emph{Key words.}}}%

\def\mscclassname{{\bfseries \emph{AMS subject classifications.}}}%
\def\mscclasses#1{\par\addvspace\medskipamount{\rightskip=0pt plus1cm
\def\and{\ifhmode\unskip\nobreak\fi\ $\cdot$
}\noindent\mscclassname\enspace\ignorespaces#1\par}}
\def\codename{{\bfseries \emph{Code.}}}%
\def\code#1{\par\addvspace\medskipamount{\rightskip=0pt plus1cm\noindent\codename\enspace\ignorespaces\url{#1}\par}}

\ifdraftmode
    \definecolor{amaranth}{rgb}{0.9, 0.17, 0.31}%
    \definecolor{americanrose}{rgb}{1.0, 0.01, 0.24}
    \colorlet{alertcolor}{amaranth}
    \colorlet{notecolor}{MidnightBlue}

    \makeatletter
    \newcommand{\todo}[1]{\marginpar{\tiny\color{alertcolor}#1}\@latex@warning{#1}\xspace}
    \makeatother
    \newcommand{\note}[1]{\marginpar{\tiny\color{notecolor}#1}}
    
    \usepackage[notref,notcite]{showkeys}
    
    \definecolor{bleudefrance}{rgb}{0.19, 0.55, 0.91}
    \newcommand{\numRevisions}{2}
    \newcommand{\revision}[2][0]{%
    \begingroup%
        \newcount\colorRatio%
        \colorRatio=\numexpr(100*(#1+1))/\numRevisions\relax%
        \colorlet{revisionColor}{bleudefrance!\the\colorRatio!black}\color{revisionColor}#2%
    \endgroup}
\else
    \newcommand{\note}[1]{}
\fi

\usepackage{pgf,tikz}
\usetikzlibrary{positioning}
\usetikzlibrary{calc}
\usetikzlibrary{intersections}
\usetikzlibrary{decorations.pathreplacing}  

\tikzset{core/.style={inner sep=0pt}}
\tikzset{contraction/.style={line width=0.75}}
\tikzset{contractionDots/.style={contraction, dotted}}

\usepackage{amsmath}             
\usepackage{amsfonts}            
\usepackage{amssymb}             
\usepackage{amsthm}              
\usepackage{thmtools}            
\usepackage[fixamsmath,disallowspaces]{mathtools} 
\mathtoolsset{showonlyrefs=true}
\usepackage{newtxtext,newtxmath} 

\colorlet{dimgray}{black!35!white}
\colorlet{lightgray}{dimgray!35!white}
\declaretheoremstyle[bodyfont=\itshape, mdframed={backgroundcolor=lightgray, linecolor=dimgray, linewidth=0.75pt, innertopmargin=1.5ex}]{claim}
\declaretheorem[style=claim]{theorem}
\declaretheorem[style=claim, numberlike=theorem]{lemma}
\declaretheorem[style=claim, numberlike=theorem]{proposition}
\declaretheorem[style=claim, numberlike=theorem]{corollary}

\declaretheoremstyle[mdframed={backgroundcolor=lightgray, linecolor=dimgray, linewidth=0.75pt, innertopmargin=1.5ex}]{definition}
\declaretheorem[style=definition, numberlike=theorem]{definition}
\declaretheoremstyle[bodyfont=\itshape, mdframed={backgroundcolor=white, linecolor=dimgray, linewidth=0.75pt, innertopmargin=1.5ex}]{remark}
\declaretheorem[style=remark, numberlike=theorem]{remark}
\declaretheoremstyle[mdframed={backgroundcolor=white, linecolor=dimgray, linewidth=0.75pt, innertopmargin=1.5ex}]{example}
\declaretheorem[style=example, numberlike=theorem]{example}

\newcommand{\indep}{\perp\kern-0.6em\perp}
\DeclareMathOperator{\supp}{supp}

\DeclareMathOperator{\rch}{rch}
\DeclareMathOperator{\cl}{cl}

\newcommand*{\dd}{\ensuremath{\mathrm{d}}}
\newcommand*{\dx}[1][x]{\ensuremath{\,\dd{#1}}}

\newcommand*{\mbb}[1]{\mathbb{#1}}
\newcommand*{\mcal}[1]{\mathcal{#1}}
\newcommand*{\mfrak}[1]{\mathfrak{#1}}

\let\inf\relax  
\DeclareMathOperator*{\inf}{inf\vphantom{\sup}}
\DeclareMathOperator*{\argmin}{arg\,min}

\DeclareMathOperator*{\esssup}{ess\,sup}

\DeclarePairedDelimiter{\pars}{\ensuremath{(}}{\ensuremath{)}}
\DeclarePairedDelimiter{\bracs}{\ensuremath{[}}{\ensuremath{]}}
\DeclarePairedDelimiter{\braces}{\ensuremath{\{}}{\ensuremath{\}}}

\DeclarePairedDelimiter{\inner}{\langle}{\rangle}
\DeclarePairedDelimiter{\norm}{\|}{\|}
\DeclarePairedDelimiter{\abs}{\lvert}{\rvert}

\makeatletter
\newcommand{\opnorm}{\@ifstar\@opnorms\@opnorm}
\newcommand{\@opnorms}[1]{%
  \left|\mkern-1.5mu\left|\mkern-1.5mu\left|
   #1
  \right|\mkern-1.5mu\right|\mkern-1.5mu\right|
}
\newcommand{\@opnorm}[2][]{%
  \mathopen{#1|\mkern-1.5mu#1|\mkern-1.5mu#1|}
  #2
  \mathclose{#1|\mkern-1.5mu#1|\mkern-1.5mu#1|}
}
\makeatother

\mathchardef\mhyphen="2D

\usepackage{dsfont}              

\let\oldbullet\bullet
\newlength{\raisebulletlen}
\setbox1=\hbox{$\bullet$}\setbox2=\hbox{\tiny$\bullet$}
\renewcommand\bullet{\raisebox{\raisebulletlen}{\,\tiny$\oldbullet$}\,}

\renewcommand{\vec}{\boldsymbol}

\makeatletter
\newcommand*{\rom}[1]{\expandafter\@slowromancap\romannumeral #1@}
\makeatother

\DeclarePairedDelimiterX\Set[1]\{\}{%
  #1%
}

\def\multiset#1#2{\ensuremath{\left(\kern-.3em\left(\genfrac{}{}{0pt}{}{#1}{#2}\right)\kern-.3em\right)}}

\newlist{thmenum}{enumerate}{1}  
\setlist[thmenum]{label=\thethmenumi., ref=\thetheorem.\thethmenumi}  

\title{\thisTitle} 
\renewcommand{\undertitle}{}
\renewcommand{\shorttitle}{\thisShorttitle}

\date{}

\author{
\href{https://orcid.org/0000-0002-2995-126X}{\includegraphics[scale=0.06]{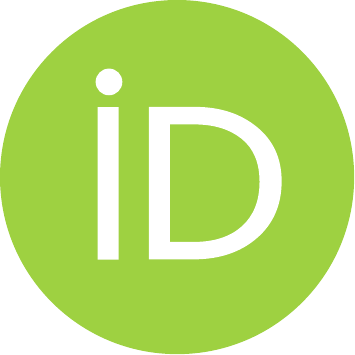}\hspace{1mm}\textcolor{black}{Philipp Trunschke\thanks{Corresponding Author. Website: \url{https://ptrunschke.github.io/}}}} \\
    Centrale Nantes \\
    Nantes Universit\'e \\
    Laboratoire de Math\'ematiques Jean Leray \\
    CNRS UMR 6629 \\
    France \\
    \href{mailto:philipp.trunschke@univ-nantes.fr}{\texttt{philipp.trunschke@univ-nantes.fr}} \\
}

\begin{document}
\maketitle

\begin{abstract}
    We consider the problem of approximating a function in a general nonlinear subset of $L^2$, when only a weighted Monte Carlo estimate of the $L^2$-norm can be computed.
    Of particular interest in this setting is the concept of sample complexity, the number of sample points that are necessary to achieve a prescribed error with high probability.
    Reasonable worst-case bounds for this quantity exist only for particular model classes, like linear spaces or sets of sparse vectors.
    For more general sets, like tensor networks or neural networks, the currently existing bounds are very pessimistic.
    
    By restricting the model class to a neighbourhood of the best approximation, we can derive improved worst-case bounds for the sample complexity.
    When the considered neighbourhood is a manifold with positive local reach, its sample complexity can be estimated by means of the sample complexities of the tangent and normal spaces and the manifold's curvature.
\end{abstract}

\keywords{least squares approximation \and sample complexity \and tensor networks \and manifolds \and positive reach}
\vspace{-1em}
\mscclasses{15A69 \and 41A30 \and 62J02 \and 65Y20 \and 68Q25}

\section{Introduction}

We consider the task of estimating an unknown function from noiseless observations.
For this problem to be well-posed we assume that the sought function can be well approximated in a given nonlinear model class.
Of particular interest in this setting is how well a sample-based estimator can approximate the sought function.
In investigating this question, many papers rely on a \emph{restricted isometry property} (RIP) or a RIP-like condition.
The RIP asserts, that the sample-based estimate of the approximation error is equivalent to the approximation error for all elements of the model class.
Without this equivalence
it is easy to conceive circumstances under which a minimizer of the empirical approximation error is arbitrarily far away from the real best approximation.
Even though the majority of compressed sensing results do not bound the probability of the RIP in our general setting,
specialised bounds exist
for linear spaces (see e.g.~\cite{cohen_2017_least-squares}) and sparse-grid spaces (see e.g.~\cite{bohn_2018_sparse_grid}), for sparse vectors (see e.g.~\cite{candes_2006_stable_recovery,rauhut_2016_weighted_l1}), low-rank matrices and tensors (see e.g.~\cite{recht_2010_nuclear_norm_minimization,rauhut_2017_iht}), as well as for neural networks (see~\cite{goessmann2020restricted}) and, only recently, for generic, non-linear model classes~\cite{eigel_2022_convergence}.

The present manuscript is not primarily concerned with providing bounds for the probability of the RIP for certain model classes.
Instead, it provides rigorous proofs for multiple, intuitively true, properties of the RIP for a large family of model classes.
The results in this paper are valid for arbitrary non-linear model classes and extends the work~\cite{eigel_2022_convergence} in three significant ways:
\begin{itemize}
    \item In~\cite{eigel_2022_convergence} the authors use the concept of a \emph{variation constant} to provide bounds for the probability of the RIP.
    However, concrete values for this constant are computed only for three special cases.
    In Theorem~\ref{thm:K_properties} we provide calculus rules that allow for the computation and estimation of the variation constant for arbitrary model classes.
    \item One model class that is considered in~\cite{eigel_2022_convergence} is that of tensor networks.
    But since the variation constant of this model class suffers from the curse of dimensionality, the work~\cite{eigel_2022_convergence} was not able to explain the absence of the curse in many practical applications~\cite{ESTW19,kraemer2015ttpde,Kressner2014}.
    The authors assumed that this was due to their estimate and conjectured the variation constant could be bounded more tightly.
    We prove this conjecture to be wrong in Theorem~\ref{thm:K_TN}.
    \item Theorems~\ref{thm:Kloc_upper} and~\ref{thm:KU_manifold_limit} provide tight upper and lower bounds for the variation constant when the model class is restricted to a small neighbourhood of the best approximation.
    These bounds do not suffer from the curse of dimensionality and thereby provide an intuition for why properly initialised gradient descent methods work so well in practice.
    We see these bounds as a first step in the analysis of gradient descent methods in the context of the RIP.
\end{itemize}

Although applicable to a wide range of model classes, our deliberations are motivated by model classes of tensor networks.
These can be regarded as generalisations of low-rank matrices as well as a subclass of neural networks with identity activation and product pooling.
As such, tensor networks present one of the simplest non-trivial, non-linear model classes.
For a more comprehensive discussion the reader is referred to the survey article~\cite{Grasedyck_2013_survey} and the monograph~\cite{hackbusch_2012_book}.

In addition to this simplicity argument, there exists a practical reason to investigate tensor networks:
They are a commonly used tool in the numerics of high-dimensional parametric PDEs~\cite{eigel2018vmc0,ESTW19}, in uncertainty quantification~\cite{haberstich_2020_thesis,wolf_sebastian_low_2019}, in dynamical systems recovery~\cite{gels_multidimensional_2019,goette_2020_dynamical_laws} and recently even in computational finance~\cite{BEST21}.
But even though this model class exhibits exceptional performance in many practical applications, a theoretical foundation for this observation is still lacking.
This stands in sharp contrast to other recent results that utilise weighted sparsity~\cite{bouchot_2015_CSPG} or precisely tailored linear spaces~\cite{cohen_2021_lognormal}.
This paper sets out to rectify this.

\paragraph{Setting}
Let $\rho$ be a probability measure on some set $Y$ and define the norm $\norm{\bullet} := \norm{\bullet}_{L^2\pars{Y,\rho}}$.
Given point-evaluations $\braces{u\pars{y^i}}_{i=1}^n$ of an unknown function $u\in L^2\pars{Y,\rho}$ at points $\braces{y^i}_{i=1}^n\subseteq Y$, we want to find a (not necessarily unique) best approximation
\begin{equation}
\label{vmc-3:eq:min}
    u_{\mcal{M}} \in \argmin_{v\in \mcal{M}}\ \norm{u - v}
\end{equation}
of $u$ in a \emph{model class} $\mcal{M}\subseteq L^2\pars{Y,\rho}$.
In general, however, $u_{\mcal{M}}$ is not computable and a popular remedy is to estimate
\begin{equation}
\label{vmc-3:eq:min_emp}
    \norm{v} \approx \norm{v}_{\boldsymbol{y}} := \pars*{\frac{1}{n}\sum_{i=1}^n w\pars{y^i} \abs{v\pars{y^i}}^2}^{1/2}
    \quad\text{and}\quad
    u_{\mcal{M}} \approx u_{\mcal{M},\boldsymbol{y}} \in \argmin_{v\in\mcal{M}}\ \norm{u - v}_{\boldsymbol{y}} ,
\end{equation}
where we assume that $w$ is a \emph{weight function}, satisfying $w > 0$ almost everywhere and $\int_Y w^{-1} \dd\rho = 1$
and where every sample point $y^i$ is drawn independently from the distribution $w^{-1}\rho$.
By considering the model class $L\mcal{M}$ for a linear operator $L$, even the minimisation of residuals of the form $\norm{u-Lv}$ can be represented in this setting.
The problems that can be phrased this way are hence so ubiquitous that we refrain from listing them.

To ensure that the optimisation problem~\eqref{vmc-3:eq:min_emp} is well-defined, we define the $w$-adapted normed linear space
\begin{equation}
    \mcal{V} := \mcal{V}_{w,\infty} := \braces{v\in L^\infty\pars{Y,\rho} : \norm{v}_{w,\infty} < \infty}
    \quad\text{with}\quad
    \norm{v}_{w,\infty} := \esssup_{y\in Y} \sqrt{w\pars{y}} \abs{v\pars{y}}
\end{equation}
and require that $u\in\mcal{V}$, $\mcal{M}\subseteq\mcal{V}$ and that point-evaluations $v\pars{y}$ exist for every function $v\in\mcal{M}$ and all $y\in Y$.
In this setting, we can prove that the empirical best approximation error $\norm{u - u_{\mcal{M},\boldsymbol{y}}}$ is, with high probability, bounded by the best approximation error $\norm{u - u_{\mcal{M}}}_{w,\infty}$.
To do this, we require the \emph{restricted isometry property} (RIP)
\begin{equation}\label{vmc-3:eq:rip}
    \operatorname{RIP}_A\pars{\delta} :\Leftrightarrow \pars{1-\delta}\norm{u}^2 \le \norm{u}_{\boldsymbol{y}}^2 \le \pars{1+\delta}\norm{u}^2 \qquad \forall u\in A
\end{equation}
to be satisfied for the set $A=\braces{u_{\mcal{M}}} - \mcal{M}$ and some $\delta \in \pars{0,1}$.

\paragraph{Structure}
The remainder of the paper is organised as follows.

Section~\ref{sec:convergence_bounds} first recalls a basic bound for the empirical best approximation error $\norm{u-u_{\mcal{M},\boldsymbol{y}}}$.
Since this bound requires the RIP to hold, the subsequent Theorem~\ref{thm:P_RIP} bounds the probability of this event.
The resulting probability depends on the reciprocal of a generalised Christoffel function, termed the \emph{variation function}.
Rules for the computation of this function are derived and applied to model classes of tensor networks.
In this way it is shown that the derived worst-case estimate for the sample complexity of tensor networks behaves asymptotically the same as the sample complexity estimate for the full tensor space in which they are contained.
This agrees with observations from matrix and tensor completion where a low-rank matrix or tensor has to satisfy an additional \emph{incoherence condition} to guarantee a reduced sample complexity~\cite{Candes2010convexRelaxationMatrixCompletion,yuan2015tensor_completion}.
Since the model class may be very large and contain ``highly coherent'' elements which can not be recovered with a low sample complexity, we consider the restriction of the model class to a neighbourhood of the best approximation $u_{\mcal{M}}$.

Section~\ref{sec:local_variation_constant} derives upper and lower bounds for the variation function of neighbourhoods of $u_{\mcal{M}}$ that form manifolds with positive local reach.
These bounds relate the variation function of the neighbourhood to those of the tangent and normal spaces at $u_{\mcal{M}}$ as well as the neighbourhood's curvature.

Section~\ref{sec:discussion} concludes this work by discussing the results and possible ways forward.

\paragraph{Notation}
We denote the set of integers from $1$ through $d$ by $\bracs{d}$ and 
define $\operatorname{supp}\pars{v} := \braces{j\in\bracs{d} : v_j\ne0}$ for any $v\in\mbb{R}^d$.

For a set $X$, $\mfrak{P}\pars{X}$ denotes the set of all subsets of $X$.
If $\pars{X,p}$ is a metric space and $Y\subseteq X$, then $\cl\pars{Y}$ denotes the closure of $Y$ in $X$ and $\mfrak{C}\pars{X}$ denotes the set of all non-empty, compact subsets of $X$.
$S\pars{x,r}\subseteq X$ denotes the sphere and $B\pars{x,r}\subseteq X$ denotes the  open ball of radius $r>0$ and with center $x\in X$.
Since the metric space $X$ should always be clear from context, we do not include it in the notation for $S\pars{x,r}$ and $B\pars{x,r}$.
If $X$ is subset of a linear space, the notation $\inner{X}$ denotes the linear span of $X$.
Finally, if $X$ is a $\mcal{C}^1$ submanifold of Euclidean space and $x\in X$ then $\mbb{T}_x X$ shall denote the tangent space of $X$ at $x$ and $\mbb{T}_x^\perp X$ shall denote its orthogonal complement in $\inner{X}$.

In Theorem~\ref{thm:K_properties} we require the concept of a continuous function that operates on sets.
The relevant topologies are induced by the following two metrics.
\begin{definition}[Hausdorff distance]
    Let $\pars{M, d}$ be a metric space.
    The function $d_{\mathrm{H}} : \mfrak{P}\pars{M}^2 \to [0,\infty)$
    \begin{equation}
        d_{\mathrm{H}}\pars{X,Y} := \max\braces{\sup_{x\in X}\inf_{y\in Y} d\pars{x,y}, \sup_{y\in Y}\inf_{x\in X} d\pars{x,y}}
    \end{equation}
    defines a metric on $\mfrak{C}\pars{M}$ and a pseudometric on $\mfrak{P}\pars{M}$.
\end{definition}

Note that this is not an appropriate metric when the sets $X$ and $Y$ are cones, since in this case
\begin{equation}
    d_{\mathrm{H}}\pars{X,Y} = \begin{cases}
                                   0 & \cl\pars{X} = \cl\pars{Y} \\
                                   \infty & \text{otherwise}
                               \end{cases} .
\end{equation}
In the following we define $\operatorname{Cone}\pars{X} := \braces{\lambda x : \lambda >0, x\in X}$ and denote by $\operatorname{Cone}\pars{\mfrak{P}\pars{M}}$ the set of all cones in $M$.
Since conic set are uniquely defined by their intersection with the unit sphere we can define a more suitable (pseudo-)metric by considering the Hausdorff distance between these intersections.
This metric is called the truncated Hausdorff distance~\cite{seeger_2010_Hausdorff_distance}.
\begin{definition}[truncated Hausdorff distance]
    The function $d_{\mathrm{tH}} : \operatorname{Cone}\pars{\mfrak{P}\pars{M}}^2 \to [0,\infty)$, defined by
    \begin{equation}
        d_{\mathrm{tH}}\pars{X,Y} := d_{\mathrm{H}}\pars{S\pars{0,1}\cap X, S\pars{0,1}\cap Y},
    \end{equation}
    is a pseudometric on $\operatorname{Cone}\pars{\mfrak{P}\pars{M}}$.
\end{definition}

\section{Worst case bounds}\label{sec:convergence_bounds}

Using the restricted isometry property, defined in equation~\eqref{vmc-3:eq:rip}, we can bound the error $\norm{u-u_{\mcal{M},\boldsymbol{y}}}$ by $\norm{u-u_{\mcal{M}}}_{\infty}$.

\begin{theorem}[Theorem~2.12 in~\cite{eigel_2022_convergence}]
\label{thm:empirical_projection_error}
    Let $u\in\mcal{V}$ and recall the definitions~\eqref{vmc-3:eq:min},~\eqref{vmc-3:eq:min_emp} and~\eqref{vmc-3:eq:rip}.
    %
    If $\operatorname{RIP}_{\braces{u_{\mcal{M}}}-\mcal{M}}\pars{\delta}$ holds, then
    \begin{equation*}
        \norm{u - u_{\mcal{M},\boldsymbol{y}}}
        \le \pars*{1+\tfrac{2}{\sqrt{1-\delta}}}\norm{u - u_{\mcal{M}}}_{w,\infty}\ .
    \end{equation*}
\end{theorem}

As defined in equation~\eqref{vmc-3:eq:rip}, the RIP can be considered for any deterministic choice of $\boldsymbol{y}\in Y^n$.
But since we assume that $\boldsymbol{y}$ is a vector of i.i.d.\ random variables, the RIP is a random event and Theorem~\ref{thm:empirical_projection_error} bounds the error of the empirical best approximation $u_{\mcal{M},\boldsymbol{y}}$ under the condition that $\operatorname{RIP}_{\braces{u_{\mcal{M}}}-\mcal{M}}$ holds.
The probability of this event can be bounded by a standard concentration of measure argument.
To do this, we define the \emph{variation function} of a subset $A\subseteq\mcal{V}$ by
\begin{equation} \label{eq:variation_constant}
    \mfrak{K}_A\pars{y} := \sup_{a\in U\pars{A}} \abs{a\pars{y}}^2
    \quad\text{where}\quad
    U\pars{A} := \braces*{\tfrac{u}{\norm{u}} : u\in A\!\setminus\!\braces{0}} .
\end{equation}
With this, we can state the following bound on the probability of $\operatorname{RIP}_A\pars{\delta}$.
\begin{theorem}[Theorem~2.7 and Corollary~2.10 in~\cite{eigel_2022_convergence}] \label{thm:P_RIP}
    For any $A\subseteq \mcal{V}$ and $\delta>0$, there exists a constant $C>0$, independent of $n$, such that
    \begin{equation*}
        \mbb{P}\bracs{\operatorname{RIP}_A\pars{\delta}} \ge 1 - C \exp\pars*{-\tfrac{n}{2}\pars{\tfrac{\delta}{\norm{\mfrak{K}_A}_{w,\infty}}}^2} .
    \end{equation*}
    If $\dim\pars{\inner{A}}<\infty$, then $C$ depends only polynomially on $\delta$ and $\norm{\mfrak{K}_A}_{w,\infty}^{-1}$.
\end{theorem}

Combining Theorem~\ref{thm:empirical_projection_error} and Theorem~\ref{thm:P_RIP} allows us to bound the probability with which the best approximation $u_{\mcal{M}}$ of a function $u$ may be recovered exactly in a given model class $\mcal{M}$.
But they can also be used to provide worst-case bounds for deterministic algorithms.
If $\norm{\mfrak{K}_{\mcal{M}-\mcal{M}}}_{w,\infty} < \infty$, Theorem~\ref{thm:P_RIP} ensures that we can find $n\in\mbb{N}$ and $\boldsymbol{y}\in Y^n$ such that $\operatorname{RIP}_{\mcal{M}-\mcal{M}}\pars{\delta}$.
This implies $\operatorname{RIP}_{\braces{u}-\mcal{M}}\pars{\delta}$ is satisfied for any $u\in\mcal{M}$ and Theorem~\ref{thm:empirical_projection_error} guarantees that any $u\in\mcal{M}$ can be exactly recovered by~\eqref{vmc-3:eq:min_emp} using only $n$ function evaluations.

Since the proof of this theorem relies on Hoeffding's inequality, one might expect the resulting bound to be quite rough.
But one can indeed show~\cite{Montgomery-Smith_1990_Rademacher_sums}, that it is optimal with respect to $n$ and $d$, i.e.\ that there exists a linear space $A$ and a constant $c>0$ such that
\begin{align*}
    \mbb{P}\bracs{\neg\operatorname{RIP}_A\pars{\delta}}
    &\le \exp\pars{-\tfrac{1}{2} n \delta^2} \qquad\text{and}\\
    \mbb{P}\bracs{\operatorname{RIP}_A\pars{\delta}}
    &\ge \exp\pars{- c^3 n \delta^2} / c .
\end{align*}
Moreover, the best currently known bounds for linear spaces~\cite{cohen_2017_least-squares} and sparse vectors~\cite{rauhut_2016_weighted_l1} also depend on the norm $\norm{\mfrak{K}_A}_{w,\infty}$.

Theorem~\ref{thm:P_RIP} guarantees that the probability of $\operatorname{RIP}_A\pars{\delta}$ increases, when the value of $\norm{\mfrak{K}_A}_{w,\infty}$ decreases.
In this way the variation function allow us to compute the optimal sampling density $w^{-1}$ for a given set $A$.
This was already proposed for linear spaces in~\cite{cohen_2017_least-squares} and
a generalisation of this construction is stated in the subsequent theorem.

\begin{theorem}[Theorem~3.1 in~\cite{eigel_2022_convergence}] \label{thm:optimal_weight_function}
    $\mfrak{K}_A$ is $\rho$-measurable and
    \begin{equation}
        \norm{w\mfrak{K}_A}_{L^\infty\pars{Y,\rho}} \ge \norm{\mfrak{K}_A}_{L^1\pars{Y,\rho}}
    \end{equation}
    for any weight function $w$.
    The lower bound is attained by the weight function $w=\frac{\norm{\mfrak{K}_A}_{L^1\pars{Y,\rho}}}{\mfrak{K}_A}$.
\end{theorem}
The fact that the optimal sampling density in Theorem~\ref{thm:optimal_weight_function} is a scaled version of $\mfrak{K}_A$ is not surprising, since $\mfrak{K}_A$ is the point-wise supremum of the densities $p_a\pars{y} := \frac{\abs{a\pars{y}}^2}{\norm{a}^2}$ over all $a\in A$.
These are the optimal importance sampling densities for the integration of $\norm{a}^2 = \int_Y a^2 \dd\rho$, minimising the variance of the empirical estimate.
The subsequent theorem provide us with rules for the explicit computation and numerical approximation of $\mfrak{K}_A$ for arbitrary sets $A$.
These results will be frequently used in the remainder of this work.
\begin{theorem}[Basic properties of $\mfrak{K}$] \label{thm:K_properties}
    Let $A,B\subseteq\mcal{V}_{w,\infty}$ and $\mcal{A}\subseteq\mfrak{P}\pars{\mcal{V}_{w,\infty}}$.
    Then the following statements hold.
    \begin{thmenum}
        \item 
        $\mfrak{K}_{\bigcup\!\mcal{A}} = \sup_{A\in\mcal{A}} \mfrak{K}_{A}$, where $\bigcup\!\mcal{A} := \bigcup_{A\in\mcal{A}}A$.
        \label{thm:K_properties:union}
        \item $\mfrak{K}_{A} = \mfrak{K}_{\operatorname{cl}\pars{A}}$. \label{thm:K_properties:closure}
        \item $\mfrak{K}_{\bullet} : \mfrak{C}\pars{\mcal{V}_{w,\infty}\setminus \braces{0}} \to \mcal{V}_{w^2,\infty}$ is continuous with respect to the Hausdorff metric. \label{thm:K_properties:continuity:compact_sets}
        \item $\mfrak{K}_{\bullet} : \mfrak{P}\pars{\mcal{V}_{w,\infty}\setminus B\pars{0,r}} \to \mcal{V}_{w^2,\infty}$ is continuous with respect to the Hausdorff pseudometric for all $r>0$. \label{thm:K_properties:continuity:all_sets}
        \item $\mfrak{K}_{\bullet} : \operatorname{Cone}\pars{\mfrak{P}\pars{\mcal{V}_{w,\infty}}} \to \mcal{V}_{w^2,\infty}$ is continuous with respect to the truncated Hausdorff pseudometric. \label{thm:K_properties:continuity:cones}
        \item If $A\perp B$, then $\mfrak{K}_{A + B} \le \mfrak{K}_{A} + \mfrak{K}_{B}$. If $A$ and $B$ are linear spaces, then equality holds. \label{thm:K_properties:sum}
        \item If $A\indep B$, then $\mfrak{K}_{A\cdot B} = \mfrak{K}_{A} \cdot \mfrak{K}_{B}$. \label{thm:K_properties:product}
        \item If $A$ and $B$ are linear spaces, then $\mfrak{K}_{A\otimes B} = \mfrak{K}_{A} \cdot \mfrak{K}_{B}$. \label{thm:K_properties:tensor_product}
    \end{thmenum}
    Here the sum ($+$), product ($\cdot$), the stochastic independence ($\indep$), and the $L^2\pars{Y,\rho}$\-orthogonality ($\perp$) of the sets $A$ and $B$ have to be understood element-wise.
\end{theorem}
The proof of this theorem is rather technical in some parts and can be found in Appendix~\ref{proof:thm:K_properties}.
As a consequence of Theorem~\ref{thm:K_properties:union} it follows that the function $\mfrak{K}_{\bullet}$ is monotonic, i.e.\ $A\subseteq B$ implies $\mfrak{K}_A \le \mfrak{K}_B$.
This and the continuity of $\mfrak{K}$ allows estimating the variation constant numerically.
Moreover, the properties of the variation function also induce analogous properties of the norm $\norm{\mfrak{K}_A}_{w,\infty}$.
Theorem~\ref{thm:K_properties:sum} for example, implies that for any $d$-dimensional linear space $A$ spanned by the $L^2\pars{Y,\rho}$-orthonormal basis $\braces{B_k}_{k\in\bracs{d}}$,
\begin{equation} \label{eq:KU_properties:dim}
    \norm{\mfrak{K}_{A}}_{w,\infty}
	= \norm{\mfrak{K}_{\inner{B_1}\oplus\cdots\oplus\inner{B_d}}}
	\le \sum_{k\in\bracs{d}} \norm{B_k}_{w,\infty}^2 .
\end{equation}

\begin{remark}
	A common misconception is that the probability bound in Theorem~\ref{thm:P_RIP} relies primarily on the metric entropy (cf.~\cite{cockreham_2017_reach}) of the model class.
    This however is not true.
    To see that $\norm{\mfrak{K}_A}_{w,\infty}$ can not be controlled by the metric entropy of $A$, consider any set $B$ for which $U\pars{B}$ is compact.
	By continuity, there exists $b^*\in B$ such that $\norm{\mfrak{K}_{\braces{b^*}}}_{w,\infty} \ge \norm{\mfrak{K}_{\braces{b}}}_{w,\infty}$ for all $b\in B$.
	Thus, $\norm{\mfrak{K}_{\braces{b^*}}}_{w,\infty} \ge \norm{\mfrak{K}_{A}}_{w,\infty}$ for any subset $A\subseteq B$, independent of its metric entropy.
\end{remark}

We use the remainder of this section to relate the variation function associated with a generic model class of tensor networks $\mcal{M}$ to the variation function of its ambient space $\inner{\mcal{M}}$.
Doing this rigorously requires the concept of rank--$1$ tensors.
Given vector spaces $\mcal{V}_1,\ldots,\mcal{V}_M$, we define the set of \emph{rank--$1$ tensors} (cf.~\cite{hitchcock1927cp_format}) as $$\mcal{T}_{1} := \braces{v_1\otimes\cdots\otimes v_M : v_m\in \mcal{V}_m\text{ for all }m} = \mcal{V}_1\cdots\mcal{V}_M,$$ i.e.\ the set of element-wise products of functions in $\mcal{V}_1, \ldots, \mcal{V}_M$.
\begin{theorem}
\label{thm:K_TN}
    Let $\mcal{M}$ be conic, symmetric and satisfy $\mcal{T}_1 \subseteq \operatorname{cl}\pars{\mcal{M}} \subseteq \mcal{V}_1\otimes\cdots\otimes\mcal{V}_M$.
    Then $$\mfrak{K}_{\braces{u_{\mcal{M}}} - \mcal{M}} = \mfrak{K}_{\braces{u_{\inner{\mcal{M}}}} - \inner{\mcal{M}}} .$$
\end{theorem}
\begin{proof}
    The claim follows from the sequence of inequalities
    \begin{equation*}
        \mfrak{K}_{\braces{u_{\inner{\mcal{M}}}} - \inner{\mcal{M}}}
        = \mfrak{K}_{\braces{u_{\mcal{M}}} - \inner{\mcal{M}}}
        \ge \mfrak{K}_{\braces{u_{\mcal{M}}} - \mcal{M}}
        \ge \mfrak{K}_{\mcal{M}}
        = \mfrak{K}_{\inner{\mcal{M}}}
        = \mfrak{K}_{\braces{u_{\inner{\mcal{M}}}} - \inner{\mcal{M}}} .
    \end{equation*}
    The first and the last equality hold, since $\braces{u_{\inner{\mcal{M}}}} - \inner{\mcal{M}} = \inner{\mcal{M}} = \braces{u_{\mcal{M}}} - \inner{\mcal{M}}$.
    The remaining relations follow from Theorem~\ref{thm:K_properties:union}, and the subsequent Lemmas~\ref{lem:KuM=KspanuM} and~\ref{lem:KM=KspanM}.
\end{proof}

Since $\norm{\mfrak{K}_{\braces{u_{\inner{\mcal{M}}}} - \inner{\mcal{M}}}}_{w,\infty} = \norm{\mfrak{K}_{\inner{\mcal{M}}}}_{w,\infty} \ge \operatorname{dim}\pars{\inner{\mcal{M}}}$ grows exponentially, by equation~\eqref{eq:KU_properties:dim}, Theorem~\ref{thm:K_TN}
implies that the probability of $\operatorname{RIP}_{\braces{u_{\mcal{M}}}-\mcal{M}}\pars{\delta}$ suffers from the curse of dimensionality.
Numerical evidence for this claim is provided in Figure~\ref{fig:recovery_phase_diagram}. 
Although the number of parameters of $\mcal{M}$ is typically much smaller than $\operatorname{dim}\pars{\inner{\mcal{M}}}$,
the observed behaviour is not surprising.
In the setting of low-rank matrix and tensor recovery it is well known, that the sought tensor has to satisfy an additional incoherence condition to be recoverable with few sample points (cf.~\cite{Candes2010convexRelaxationMatrixCompletion,yuan2015tensor_completion}).

\begin{figure}[htb]
    \centering
    \begin{subfigure}[t]{0.5\textwidth}
        \centering
        \includegraphics[width=\textwidth]{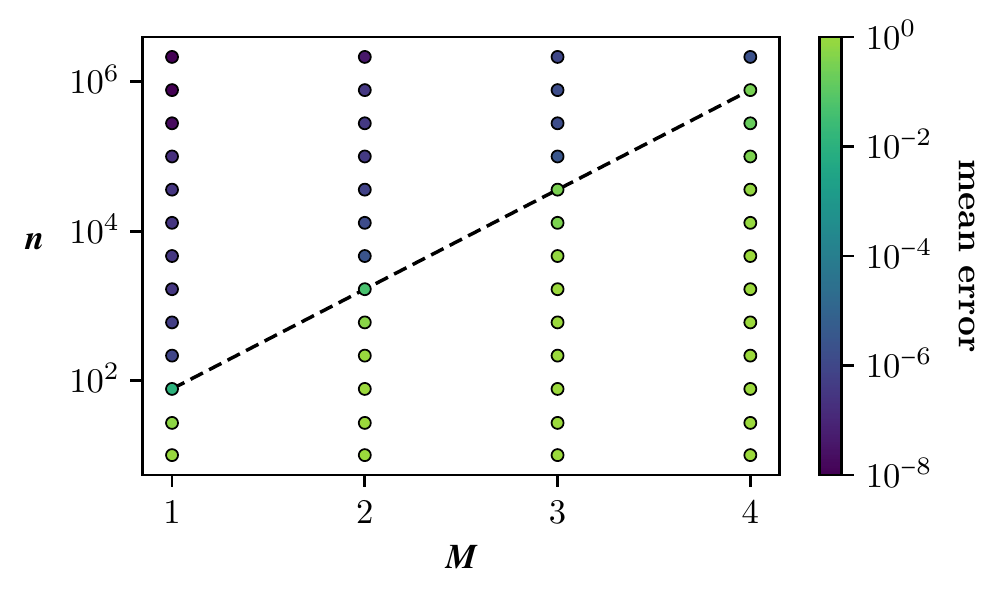}
        \caption{$u$ is defined by $C_{k_1,...,k_M} = 1$.}
        \label{fig:recovery_phase_diagram:worst_case}
    \end{subfigure}
    \begin{subfigure}[t]{0.5\textwidth}
        \centering
        \includegraphics[width=\textwidth]{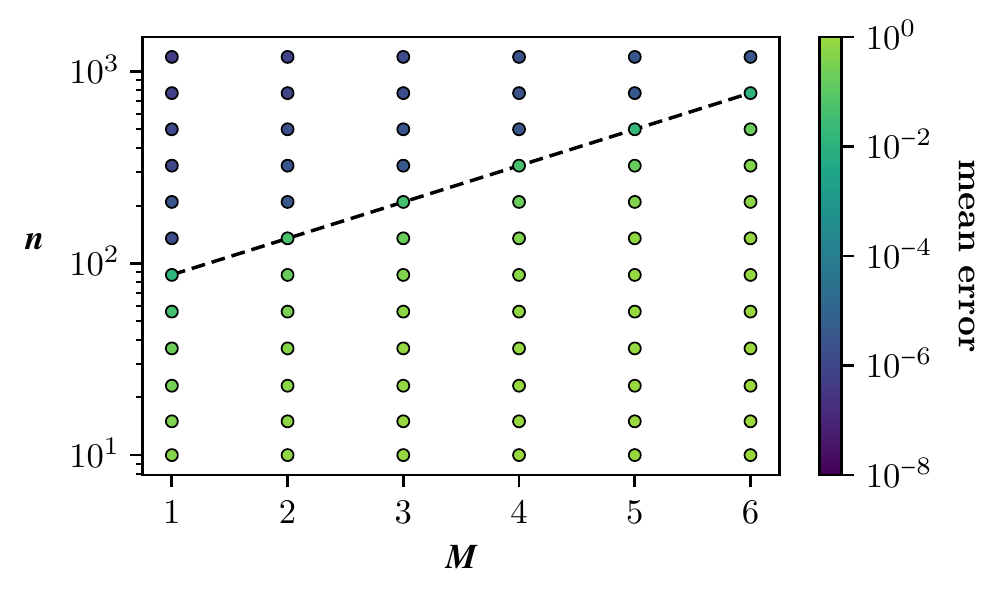}
        \caption{$u\pars{y} := \exp\pars{y_1+\cdots+y_M}$.}
        \label{fig:recovery_phase_diagram:exp}
    \end{subfigure}
    \caption{Two phase diagrams for the empirical approximation of a function $u$ in the tensor product basis of Legendre polynomials.
    The coefficient tensor $C\in\pars{\mbb{R}^{15}}^{\otimes M}$ of the best approximation is of rank $1$ in both examples.
    For every order $M$ and number of samples $n$, the mean error is computed as the relative $L^2$-error of the approximation, averaged over $20$ independent realisations.
    A hard-thresholding algorithm (cf.~\cite{eigel2018vmc0}) was used for recovery.}
    \label{fig:recovery_phase_diagram}
\end{figure}

We use the remainder of this section to prove the Lemmas~\ref{lem:KuM=KspanuM} and~\ref{lem:KM=KspanM}, that are used in the proof of Theorem~\ref{thm:K_TN}.

\begin{lemma} \label{lem:closure_of_sum}
    $\cl\pars{A+B} \supseteq \cl\pars{A}+\cl\pars{B}$ for all sets $A$ and $B$.
\end{lemma}
\begin{proof}
    Let $a\in \cl\pars{A}$ and $b\in \cl\pars{B}$.
    Then there exist sequences $\braces{a_k}\in A$ and $\braces{b_k}\in B$ such that $a_k\to a$ and $b_k \to b$.
	Since $a_k + b_k\in A+B$ we have $a+b = \lim_{k\to\infty} \pars{a_k+b_k} \in \cl\pars{A+B}$.
\end{proof}

\begin{lemma} \label{lem:KuM=KspanuM}
    Let $\mcal{M}$ be conic and symmetric and let $v\in\mcal{V}$.
    Then $\mfrak{K}_{\braces{v}-\mcal{M}} = \mfrak{K}_{\inner{v}+\cl\pars{\mcal{M}}}$.
\end{lemma}
\begin{proof}
	By Theorem~\ref{thm:K_properties:union}, Lemma~\ref{lem:closure_of_sum} and Theorem~\ref{thm:K_properties:closure}, it holds that
	\begin{equation}
		\mfrak{K}_{\braces{v}-\mcal{M}}
		\le \mfrak{K}_{\inner{v}-\cl\pars{\mcal{M}}}
		\le \mfrak{K}_{\cl\pars{\operatorname{Cone}\pars{\braces{-v,v}}-\mcal{M}}}
		= \mfrak{K}_{\operatorname{Cone}\pars{\braces{-v,v}}-\mcal{M}} .
	\end{equation}
	Since $\mcal{M}$ is conic and symmetric, it holds that
	\begin{equation}
		\operatorname{Cone}\pars{\braces{-v,v}}-\mcal{M} = \operatorname{Cone}\pars{\braces{-v,v}-\mcal{M}} = \operatorname{Cone}\pars{-\pars{\braces{v}-\mcal{M}}} \cup \operatorname{Cone}\pars{\braces{v}-\mcal{M}}
	\end{equation}
	and consequently $\mfrak{K}_{\operatorname{Cone}\pars{\braces{-v,v}}-\mcal{M}} = 
	\max\braces{\mfrak{K}_{\operatorname{Cone}\pars{-\pars{\braces{v}-\mcal{M}}}}, \mfrak{K}_{\operatorname{Cone}\pars{\braces{v}-\mcal{M}}}}$.
	Since Theorem~\ref{thm:K_properties:product} implies that $\mfrak{K}_{C A} = \mfrak{K}_A$ for any set $A$ and all subsets $C\subseteq\mbb{R}$, it follows that $\mfrak{K}_{\operatorname{Cone}\pars{\braces{-v,v}}-\mcal{M}} = \mfrak{K}_{\braces{v}-\mcal{M}}$.
\end{proof}


\begin{lemma} \label{lem:KM=KspanM}
    Let
    $\mcal{M}$ satisfy 
    $\mcal{T}_1 \subseteq \operatorname{cl}\pars{\mcal{M}} \subseteq \mcal{V}_1\otimes\cdots\otimes\mcal{V}_M$.
    Then $\mfrak{K}_{\mcal{M}} = \mfrak{K}_{\inner{\mcal{M}}}$.
\end{lemma}
\begin{proof}
    Since $\operatorname{cl}\pars{\mcal{M}}\supseteq\mcal{T}_1$, Theorem~\ref{thm:K_properties:closure} and~\ref{thm:K_properties:union} imply
    \begin{equation}
        \mfrak{K}_{\mcal{M}} = \mfrak{K}_{\operatorname{cl}\pars{\mcal{M}}}\ge \mfrak{K}_{\mcal{T}_1} \label{eq:KM=KT1}
    \end{equation}
    and by Theorem~\ref{thm:K_properties:product} and \ref{thm:K_properties:tensor_product}, it holds that 
    \begin{equation}
        \mfrak{K}_{\mcal{T}_1} = \mfrak{K}_{\mcal{V}_1\cdots\mcal{V}_M} = \mfrak{K}_{\mcal{V}_1} \cdots \mfrak{K}_{\mcal{V}_M} = \mfrak{K}_{\mcal{V}_1\otimes\cdots\otimes\mcal{V}_M} = \mfrak{K}_{\inner{\mcal{M}}} . \label{eq:KT1=KV}
    \end{equation}
	Combining equation~\eqref{eq:KM=KT1} and~\eqref{eq:KT1=KV} and employing Theorem~\ref{thm:K_properties:union} a final time yields the chain of inequalities
    $\mfrak{K}_{\inner{\mcal{M}}} \ge \mfrak{K}_{\mcal{M}} \ge \mfrak{K}_{\mcal{T}_1} = \mfrak{K}_{\inner{\mcal{M}}}$,
    which concludes the proof.
\end{proof}


\section{Restriction to neighbourhoods} \label{sec:local_variation_constant}

Theorem~\ref{thm:K_TN} illustrates a problem that is not specific to the model class of tensor networks.
Since Theorem~\ref{thm:P_RIP} provides a worst-case bound for the entire model class, it must take into account every single element of $A = \braces{u_{\mcal{M}}} - \mcal{M}$,
even those that are very irregular and far away from $u_{\mcal{M}}$.
%
In light of monotonicity of $\mfrak{K}_{\bullet}$ (Theorem~\ref{thm:K_properties:union}), it stands to reason that this problem could be eliminated by restricting the model class $\mcal{M}$ to a small neighbourhood $\mcal{N}\subseteq\mcal{M}$ of $u_{\mcal{M}}$.

For model classes that can be locally linearised in $u_{\mcal{M}}$ we can compute upper and lower bounds for $\mfrak{K}_{\braces{u_{\mcal{M}}} - \mcal{N}}$.
To formalise this we utilise the concept of reach.
First introduced by Federer~\cite{federer_1959_curvature_measures}, the local reach $\rch\pars{A, v}$ of a set $A\subseteq\mcal{V}$ at a point $v\in A$ is the largest number such that any point at distance less than $\rch\pars{A, v}$ from $v$ has a unique nearest point in $A$.
\begin{definition}
    For any subset $A\subseteq\mcal{V}$, let $P_A : \mcal{V} \to A$ be the best approximation operator $P_A v := \argmin_{a\in A} \norm{v - a}$ and denote by $\operatorname{dom}\pars{P_A}$ the domain on which it is uniquely defined.
    Then $\rch\pars{A, v} := \sup\braces{r\ge 0 : B\pars{v, r}\subseteq\operatorname{dom}\pars{P_A}}$.
\end{definition}
Intuitively, if the local reach $\rch\pars{A,v}$ of a set $A$ is larger than $R > 0$ at every point $v\in A$, then a ball of radius $R$ can roll freely around $A$ without ever getting stuck~\cite{cuevas_2012_rolling}.
%
%
Moreover, if $A$ is a differentiable manifold, then the rolling ball interpretation of the reach also provides a measure of the curvature for $A$ that generalises the concept of an osculating circle for planar curves.
The distance of the manifold $A$ to its tangent space $v + \mbb{T}_vA$ is hence inversely proportional to the reach.

\begin{proposition} \label{prop:nhood_properties}
    Let $\mcal{N}$ be a manifold, $u\in\mcal{N}$ and assume that $R := \rch\pars{\mcal{N}, u} > 0$.
    Then there exists $0 < r^\ast \le R$ such that the following two statements hold.
    \begin{thmenum}
    \item For all $r\in\pars{0,\infty}$ and every $v\in\mcal{N}\cap B\pars{u, r}$ there exists $w\in \pars{u+\mbb{T}_u\mcal{N}}\cap B\pars{u, r}$ with $\norm{v - w} \le \frac{\norm{u-v}^2}{2R}$.
	\label{prop:nhood_properties:tangent_space_projection}
    \item For all $r\in\pars{0,r^\star}$ and every $w\in \pars{u+\mbb{T}_u\mcal{N}}\cap B\pars{u, r}$ there exists $v\in \mcal{N}\cap B\pars{u, r}$ with $\norm{w - v} \le \frac{\norm{u-w}^2}{2R}$.
	\label{prop:nhood_properties:nhood_projection}
    \end{thmenum}
\end{proposition}

A very elegant proof of this statement that relies only on fundamental geometric arguments can be found in Appendix~\ref{proof:prop:nhood_properties}.
With all of this at our disposal we can now provide a first upper bound for the variation function $\mfrak{K}_{\braces{u_{\mcal{M}}} - \mcal{N}}$.

\begin{theorem}\label{thm:Kloc_upper}
    \newcommand{\niceSqrt}[1]{\sqrt{\smash[b]{#1}\vphantom{\mbb{T}_{u_{\mcal{M}}}}}}
    Let $r>0$ and $\mcal{N} \subseteq \mcal{M} \cap B\pars{u_{\mcal{M}}, r}$ be a manifold with $R := \rch\pars{\mcal{N}, u_{\mcal{M}}} \ge r$.
    Then
    \begin{equation}
        \mfrak{K}_{\braces{u_{\mcal{M}}} - \mcal{N}} \le \pars*{\niceSqrt{\mfrak{K}_{\mbb{T}_{u_{\mcal{M}}}\mcal{N}}} + \frac{r}{2R}\niceSqrt{\mfrak{K}_{\mbb{T}^\perp_{u_{\mcal{M}}}\mcal{N}}}}^2 .
    \end{equation}
\end{theorem}
\begin{proof}
    \newcommand{\tangentSpace}{\mbb{T}_{u_{\mcal{M}}}\mcal{N}}
	Recall that $\mfrak{K}_{\braces{u_{\mcal{M}}}-\mcal{N}} = \sup_{v\in\mcal{N}} \mfrak{K}_{\braces{u_{\mcal{M}}-v}}$.
    We start by bounding $\mfrak{K}_{\braces{u_{\mcal{M}}-v}}$ for a fixed $v\in\mcal{N}$.
    For this define
    \begin{equation}
        \begin{aligned}
        \alpha_T &= \norm{P \pars{u_{\mcal{M}} - v}}
        &\text{and}\quad
        b_T &= \alpha_T^{-1} P \pars{u_{\mcal{M}} - v} , \\
        \alpha_N &= \norm{Q \pars{u_{\mcal{M}} - v}}
        &\text{and}\quad
        b_N &= \alpha_N^{-1} Q \pars{u_{\mcal{M}} - v} ,
        \end{aligned}
    \end{equation}
    where $P$ denotes the orthogonal projection onto $\tangentSpace$ and $Q := 1 - P$.
    This means that $\pars{u_{\mcal{M}}-v}^2 = \pars{\alpha_Tb_T + \alpha_Nb_N}^2 \le \pars{\abs{\alpha_T}\abs{b_T} + \abs{\alpha_N}\abs{b_N}}^2$ and $\norm{u_{\mcal{M}} - v} = \sqrt{\alpha_T^2 + \alpha_N^2} \le r$.
	Since $\norm{u_{\mcal{M}}-v} \le r \le R$, Proposition~\ref{prop:nhood_properties:tangent_space_projection} implies $\abs{\alpha_N} \le \frac{\alpha_T^2 + \alpha_N^2}{2R}$ and we can bound
    \begin{equation}
        \frac{\abs{\alpha_T}}{\sqrt{\alpha_T^2 + \alpha_N^2}} \le 1
        \qquad\text{and}\qquad
        \frac{\abs{\alpha_N}}{\sqrt{\alpha_T^2 + \alpha_N^2}} \le \frac{r}{2R} .
    \end{equation}
    Consequently,
    \begin{align}
        \mfrak{K}_{\braces{u_{\mcal{M}}-v}}
        &= \frac{\pars{u_{\mcal{M}}-v}^2}{\norm{u_{\mcal{M}}-v}^2} \\
        &\le \pars{\abs{b_T} + \frac{r}{2R}\abs{b_N}}^2 \\
        &= \pars*{\sqrt{\mfrak{K}_{\braces{b_T}}} + \frac{r}{2R}\sqrt{\mfrak{K}_{\braces{b_N}}}}^2 \\
        &= \pars*{\sqrt{\mfrak{K}_{P \braces{u_{\mcal{M}}-v}}} + \frac{r}{2R}\sqrt{\mfrak{K}_{Q \braces{u_{\mcal{M}}-v}}}}^2 .
    \end{align}
    Taking the supremum over $v\in\mcal{N}$, we obtain
    \begin{align}
	    \mfrak{K}_{\braces{u_{\mcal{M}}}-\mcal{N}}
        &\le \pars*{\sqrt{\sup_{v\in\mcal{N}} \mfrak{K}_{P \braces{u_{\mcal{M}}-v}}} + \frac{r}{2R}\sqrt{\sup_{v\in\mcal{N}} \mfrak{K}_{Q \braces{u_{\mcal{M}}-v}}}}^2 \\
        &= \pars*{\sqrt{\mfrak{K}_{P \pars{\braces{u_{\mcal{M}}}-\mcal{N}}}} + \frac{r}{2R}\sqrt{\mfrak{K}_{Q \pars{\braces{u_{\mcal{M}}}-\mcal{N}}}}}^2
    \end{align}
    To conclude the proof, observe that $P \pars{\braces{u_{\mcal{M}}}-\mcal{N}} = \tangentSpace$ and $Q \pars{\braces{u_{\mcal{M}}}-\mcal{N}} \subseteq \mbb{T}^\perp_{u_{\mcal{M}}}\mcal{M}$.
\end{proof}

This theorem provides an upper bound to the variation function in a neighbourhood of the best approximation $u_{\mcal{M}}$ and supports the intuition that the probability of finding the best approximation increases when the search is restricted to a small neighbourhood of the best approximation.


We now prove that the bound from Theorem~\ref{thm:Kloc_upper} is sharp when the the radius $r$ of the neighbourhood goes to zero.
This again follows from the fact that the distance of the manifold to its tangent space can be bounded with the help of the reach.
We start by formulating the following simple corollary from Proposition~\ref{prop:nhood_properties}.
\begin{corollary} \label{cor:convergence_M}
    Let $\mcal{N}$ be a manifold and $u\in\mcal{N}$.
    Moreover, assume that $R := \rch\pars{\mcal{N}, u} > 0$.
    Then, for any $r\le R$,
    \begin{equation}
        d_{\mathrm{H}}\pars{\mcal{N}\cap B\pars{u, r}, \pars{u+\mbb{T}_u\mcal{N}}\cap B\pars{u,r}} \le \frac{r^2}{2R} .
    \end{equation}
\end{corollary}

Since both $\mcal{N}\cap B\pars{v,r}$ and $\pars{v+\mbb{T}_v\mcal{N}}\cap B\pars{v,r}$ are contained in the ball $B\pars{v,r}$, they have to converge with a rate of at least $r$.
The key message of the Corollary~\ref{cor:convergence_M} is hence the convergence at a faster rate of $r^2$.
This means that this convergence is preserved even after a rescaling of the sets.
This is made explicit with the help of the normalisation operator $U$ from equation~\eqref{eq:variation_constant} in the subsequent proposition.

\begin{proposition} \label{prop:convergence_UM}
    Let $\mcal{N}$ be a manifold and $u\in\mcal{N}$.
    Moreover, assume that $R=\rch\pars{\mcal{N}, u} > 0$.
    Then, for any $r\le R$,
    \begin{equation}
        d_{\mathrm{H}}\pars{U\pars{\mcal{N}\cap B\pars{u,r}-u}, U\pars{\mbb{T}_u\mcal{N}}} \le \frac{r}{R} .
    \end{equation}
\end{proposition}

The proof of both statements can be found in Appendix~\ref{proof:prop:convergence_UM}.
Using this convergence and the continuity of the variation function, we can deduce the following corollary.

\begin{theorem} \label{thm:KU_manifold_limit}
    Let $\mcal{N} \subseteq \mcal{M}$ be a manifold and neighbourhood of $u_{\mcal{M}}$ and assume that $\rch\pars{\mcal{N}, u_{\mcal{M}}} > 0$.
    Then
    \begin{equation}
        \mfrak{K}_{\mbb{T}_{u_{\mcal{M}}}\mcal{N}} \le \mfrak{K}_{\braces{u_{\mcal{M}}} - \mcal{N}} .
    \end{equation}
\end{theorem}
\begin{proof}
    We start by proving that 
    \begin{equation} \label{eq:KU_manifold_limit}
        \lim_{r\to 0} \mfrak{K}_{\braces{u_{\mcal{M}}} - \mcal{N}\cap B\pars{u_{\mcal{M}}, r}} = \mfrak{K}_{\mbb{T}_{u_{\mcal{M}}}\mcal{N}} .
    \end{equation}
    Recall from the proof of Theorem~\ref{thm:K_properties:continuity:compact_sets}--\ref{thm:K_properties:continuity:cones} in Appendix~\ref{proof:thm:K_properties}, that $\mfrak{K}_{\bullet} = F\circ\,U$, where $F = \operatorname{sq}\circ \sup\circ \operatorname{abs}$ and $\operatorname{sq}$, $\sup$, and $\operatorname{abs}$ are defined in Equations~\eqref{eq:sq}--\eqref{eq:sup}\vphantom{\eqref{eq:sq}\eqref{eq:abs}\eqref{eq:sup}}.
    By Lemmas~\ref{lem:abs_continuous}--\ref{lem:sq_continuous}\noeqref{eq:abs} (also stated in Appendix~\ref{proof:thm:K_properties}) $F : \mfrak{P}\pars{\mcal{V}_{w,\infty}} \to \mcal{V}_{w^2,\infty}$ is continuous with respect to the Hausdorff pseudometric.
    The continuity of $F$ and Proposition~\ref{prop:convergence_UM} then imply 
    \[
        \lim_{r\to 0} F\pars{U\pars{\mcal{N}\cap B\pars{u_{\mcal{M}},r}-u_{\mcal{M}}}}
        = F\pars{\lim_{r\to 0} U\pars{\mcal{N}\cap B\pars{u_{\mcal{M}},r}-u_{\mcal{M}}}}
        = F\pars{U\pars{\mbb{T}_{u_{\mcal{M}}}\mcal{N}}} .
    \]
    Theorem~\ref{thm:K_properties:union} implies $\mfrak{K}_{\braces{u_{\mcal{M}}} - \mcal{N}\cap B\pars{u_{\mcal{M}}, r}} \le \mfrak{K}_{\braces{u_{\mcal{M}}} - \mcal{N}}$.
    After taking the limit $r\to 0$, the assertion follows from equation~\eqref{eq:KU_manifold_limit}.
\end{proof}

Theorems~\ref{thm:Kloc_upper} and~\ref{thm:KU_manifold_limit} reveal that the regularity of the model class in close proximity to the best approximation is captured entirely by the variation function of the tangent space.
This establishes $\mfrak{K}_{\mbb{T}_{u_\mcal{M}}\mcal{M}}$ as a generalised measure of \emph{incoherence} and allows us to provide a heuristic explanation for the practical success of many low-rank approximation algorithms.
\begin{remark} \label{rmk:restriction}
    Suppose that the initial guess $v_0$ satisfies $v_0 \in \mcal{N}$ for some neighbourhood $\mcal{N}$ of $u_{\mcal{M}}$.
    Then, with high probability, every iterate $v_k$ will satisfy $v_k \in \mcal{N}$,
    if the variation function $\mfrak{K}_{\braces{u_{\mcal{M}}} - \mcal{N}}$ is sufficiently small.
    The preceding theorem gives sufficient conditions for this to happen, namely that
    \begin{enumerate}
        \item the initial guess is close to the best approximation (the neighbourhood $\mcal{N}$ is small),
	\item the best approximation is sufficiently incoherent ($\norm{\mfrak{K}_{\mbb{T}_{u_{\mcal{M}}}\mcal{M}}}_{w,\infty}$ is small) and
        \item the local reach of the manifold is positive.
    \end{enumerate}
\end{remark}

The three assumptions in Remark~\ref{rmk:restriction} are necessary but challenging to ensure in practice.
Specifically, the first assumption depends significantly on the problem and the additional information available.
We conclude this section by discussion the validity of the remaining two assumptions for linear spaces, sparse vectors and low-rank matrices in Examples~\ref{ex:polynomial_regression}--\ref{ex:matrix_completion}.
For all three model classes we compute the local reach as well as the variation function of the tangent space.

\begin{example}[Linear spaces] \label{ex:polynomial_regression}
    \leavevmode\setlength{\parskip}{0px}
    Let $Y$ be some set and $\rho$ be a probability measure on $Y$.
    Moreover, let $\mcal{M}$ be the linear space spanned by the $L^2\pars{Y,\rho}$-orthonormal basis $\braces{b_k}_{k\in\bracs{d}}$ with $d\in\mbb{N}$.
    Since the best approximation in $\mcal{M}$ is unique for any element in $\mcal{V}$, it holds that $\rch\pars{\mcal{M}, v} = \infty$ for every $v\in\mcal{M}$.
    Furthermore, for any $v\in\mcal{M}$, it holds that $\mbb{T}_v\mcal{M} = \mcal{M}$.
    The variation function of any neighbourhood $\mcal{N}\subseteq\mcal{M}$ of $v$ is hence given by
    \begin{equation}
	    \mfrak{K}_{\braces{v}-\mcal{N}}
	    = \mfrak{K}_{\braces{v}-\mcal{M}}
        = \mfrak{K}_{\mbb{T}_v\mcal{M}}
        = \mfrak{K}_{\inner{b_1}\oplus\cdots\oplus\inner{b_d}}
        = b_1^2 + \cdots + b_d^2 .
    \end{equation}
	In the special case $Y=\bracs{-1,1}$ and $\rho=\frac{1}{2}\dx$ and for the basis of Legendre polynomials, this results in the lower bound $\norm{\mfrak{K}_{\mbb{T}_v\mcal{M}}}_{w,\infty} = d^2$ for the weight function $w\equiv 1$.
	Using the optimal weight function, given by Theorem~\ref{thm:optimal_weight_function}, yields the bound $\norm{\mfrak{K}_{\mbb{T}_v\mcal{M}}}_{w,\infty} = d$.
\end{example}

\begin{example}[Sparse vectors] \label{ex:sparse_regression}
    \leavevmode\setlength{\parskip}{0px}
    Let $Y$ be some set and $\rho$ be a probability measure on $Y$.
    For $d\in\mbb{N}\cup\braces{\infty}$, let $\braces{b_k}_{k\in\bracs{d}}$ be an $L^2\pars{Y,\rho}$-orthonormal basis and let $\braces{\omega_k}_{k\in\bracs{d}}$ satisfy $\omega_k \ge \norm{b_k}_{w,\infty}$ for all $k\in\bracs{d}$.
    For ease of notation we identify every $v\in\inner{b_1,\ldots,b_d}$ with its coefficient vector $v\in\mbb{R}^d$.
    The model class of $\omega$-weighted $s$-sparse functions can then be defined as
    \begin{equation}
        \mcal{M} := \braces*{v\in\inner{b_1,\ldots,b_d} \,:\, \norm{v}_{\omega,0} \le s}
        \qquad\text{with}\qquad
        \norm{v}_{\omega,0} := \sqrt{\textstyle\sum_{k\in\operatorname{supp}\pars{v}} \omega_k^2}
        \ .
    \end{equation}
	Now consider the neighbourhood $\mcal{N} := \mcal{M}\cap B\pars{v,r}$ for some $v\in\mcal{M}$ with $\norm{v}_{\omega,0}^2 + \omega_k^2 > s^2$ for all $k\in\bracs{d}\setminus\supp\pars{v}$ and $r\le \min \braces{\abs{v_k} \,:\, k\in\supp\pars{v}}$.
    Since
    \begin{equation}
        \mcal{N}
        = \braces{w\in\inner{b_1,\ldots,b_d} \,:\, \supp\pars{w}=\supp\pars{v}\text{ and }\norm{v - w}_{2}\le r}
    \end{equation}
    is a ball of radius $r$ in a $\abs{\supp\pars{v}}$-dimensional linear space, it holds that $\rch\pars{\mcal{N}, v} = \infty$.
    We have, furthermore, that $\mbb{T}_v\mcal{N} = \inner{b_k \,:\, k\in\supp\pars{v}}$ and hence
    \begin{equation}
        \textstyle
        \mfrak{K}_{\mbb{T}_v\mcal{N}}
        = \sum_{k\in\supp\pars{v}} \norm{b_k}_{L^\infty}^2
        \le \sum_{k\in\supp\pars{v}} \omega_k^2
        \le s^2 .
    \end{equation}
    In contrast to the previous example, this already yields a mild improvement over the bound $\mfrak{K}_{\braces{u_{\mcal{M}}} - \mcal{M}} \le \mfrak{K}_{\mcal{M} - \mcal{M}} \le 2s^2$.
	Moreover, we see that the condition that $\mcal{N}$ is a manifold with $\rch\pars{\mcal{N},v} > 0$ becomes non-trivial.
	If the $\omega$-weighted sparsity $s$ is not completely exhausted by $v$, then $\mcal{N}$ may not be a manifold or the best approximation in $\mcal{N}$ may not be uniquely defined and the local reach vanishes.
	This shows, that the sparsity of the sought function $u$ should not be overestimated, when defining the model class $\mcal{M}$.
\end{example}

\begin{example}[Low-rank matrices] \label{ex:matrix_completion}
    \leavevmode\setlength{\parskip}{0px} 
    Let $Y = Y_{\mathrm{L}}\times Y_{\mathrm{R}}$ be some set and $\rho = \rho_{\mathrm{L}}\otimes\rho_{\mathrm{R}}$ be a product probability measure on $Y$.
    For $m\in\braces{\mathrm{L}, \mathrm{R}}$, let $\braces{b_{m,k}}_{k\in\bracs{d}}$ be an $L^2\pars{Y_m,\rho_m}$-orthonormal basis and let $B_{kl} := b_{\mathrm{L},k}\otimes b_{\mathrm{R},l}$ be the corresponding tensor product basis.
    As in the previous example, we identify every $v\in\inner{B_{11},\ldots,B_{dd}}$ with its coefficient matrix $v\in\mbb{R}^{d\times d}$.
    Now we can define the model class
    \begin{equation}
        \mcal{M} := \braces{v\in\inner{B_{11},\ldots,B_{dd}} : \operatorname{rank}\pars{v}\le R}
    \end{equation}
    of functions with coefficient matrices of rank at most $R$.
    Let $v\in\mcal{M}$ be of rank $R$ and consider the neighbourhood $\mcal{N} := \mcal{M}\cap B\pars{v,r}$ for some $r\le \sigma_R\pars{v}$.
    Here, $\sigma_{k}\pars{v}$ denotes the $k$\textsuperscript{th} largest singular value of $v$ and $B\pars{v,r}$ denotes the $\norm{\bullet}$-ball of radius $r$ around $v$.
	Then $\mcal{N}$ corresponds to a $\norm{\bullet}_{\mathrm{Fro}}$-ball of radius $r$ in the $C^\infty$-manifold of rank-$R$ matrices and $\operatorname{rch}\pars{\mcal{N}, v} \ge \frac{r}{2}$.
	To see this, let $w$ be any matrix with $\norm{v-w}_{\mathrm{Fro}} \le \frac{r}{2}$ and observe that
    \begin{equation}
        \sigma_{R}\pars{w}
        \ge \sigma_{R}\pars{v} - \frac{r}{2}
        > 0
        \qquad\text{and}\qquad
        \sigma_{R}\pars{w} - \sigma_{R+1}\pars{w}
	    \ge \sigma_{R}\pars{v} - \frac{r}{\sqrt{2}}
        > 0 .
    \end{equation}
    This means, that $\operatorname{rank}\pars{w} \ge R$ and that its best rank-$R$ approximation, given by the truncated singular value decomposition, is uniquely defined.
    This approximation, denoted by $x$, lies in $\mcal{N}$, since $\norm{v-x} \le \norm{v-w} + \norm{w-x} \le \frac{r}{2} + \frac{r}{2}$.
    Let $v = w_{\mathrm{L}}^{\vphantom{\intercal}} \operatorname{diag}\pars{\sigma} w_{\mathrm{R}}^\intercal$ be the singular value decomposition of $v$.
    The tangent space at $v$ can then be written as $\mbb{T}_v\mcal{N} = \mcal{W}_{\mathrm{L}} \oplus \mcal{W}_{\mathrm{R}}$ with $\mcal{W}_{\mathrm{L}} := \inner{w_{\mathrm{L}}} \otimes \inner{b_{\mathrm{R},1}, \ldots, b_{\mathrm{R},d}}$ and $\mcal{W}_{\mathrm{R}} := \inner{w_{\mathrm{L}}}^\perp \otimes \inner{w_{\mathrm{R}}}$, where $\inner{w_{\mathrm{L}}}$ denotes the linear span of the columns of $w_{\mathrm{L}}$ and $\inner{w_{\mathrm{L}}}^\perp$ denotes its orthogonal complement in $\inner{b_{\mathrm{L},1}, \ldots, b_{\mathrm{L},d}}$.
    Thus
    \begin{align}
        \mfrak{K}_{\mbb{T}_v\mcal{N}}
        &= \mfrak{K}_{\mcal{W}_{\mathrm{L}}} + \mfrak{K}_{\mcal{W}_{\mathrm{R}}}, \\
        \mfrak{K}_{\mcal{W}_{\mathrm{L}}}
        &= \mfrak{K}_{\inner{w_{\mathrm{L}}}}\mfrak{K}_{\inner{b_{\mathrm{R},1},\ldots,b_{\mathrm{R},d}}}
        \text{ and} \\
        \mfrak{K}_{\mcal{W}_{\mathrm{R}}}
        &= \mfrak{K}_{\inner{w_{\mathrm{L}}}^\perp}\mfrak{K}_{\inner{w_{\mathrm{R}}}}
        \le \mfrak{K}_{\inner{b_{\mathrm{L},1},\ldots,b_{\mathrm{L},d}}}\mfrak{K}_{\inner{w_{\mathrm{R}}}} .
    \end{align}
    For the sake of simplicity, assume that $\norm{b_{m,k}}_{w,\infty} = 1$ for all $m\in\braces{\mathrm{L},\mathrm{R}}$ and $k\in\bracs{d}$.
    Then we obtain the bound $\mfrak{K}_{\mbb{T}_v\mcal{N}} \le d\pars{\mfrak{K}_{\inner{w_{\mathrm{L}}}} + \mfrak{K}_{\inner{w_{\mathrm{R}}}}}$.
    This shows that, when $R$ is fixed, the variation function $\mfrak{K}_{\mbb{T}_v\mcal{N}}$ grows only linearly with $d$ while $\mfrak{K}_{\mcal{M}}$ grows quadratically.
    This reduction is even more pronounced for model classes of higher order tensors.
    
    This is the first example which truely benefits from restricting the model class to a neighbourhood of the best approximation.
    But as in the previous example, we see that the condition that $\mcal{N}$ is a manifold with positive local reach is not trivial.
    If the rank of $v$ is strictly smaller than $R$, then $\mcal{N}$ may not be a manifold or the best approximation in $\mcal{N}$ may not be uniquely defined and the local reach vanishes.
    This shows that the rank of $u$ must not be overestimated, when defining the model class $\mcal{M}$.
    It is easy to see, that otherwise $\mfrak{K}_{\braces{u_{\mcal{M}}} - \mcal{N}} = \mfrak{K}_{\inner{\mcal{M}}}$ for all neighbourhoods $\mcal{N}$.
    Many successful algorithms for low-rank approximation already account for this by starting with an initial guess of rank $1$ and successively increasing the rank while testing for divergence on a validation set.
\end{example}

\section{Discussion}\label{sec:discussion}

A crucial point in this paper is the observation that any reasonable class of tensor networks exhibits the same variation function as the linear space in which it is embedded.
Theorem~\ref{thm:K_TN} indicates that a feasible empirical best approximation is not possible when considering the entire model class of tensor networks and motivates our investigation of the variation function in small neighbourhoods of the best approximation.

If the neighbourhood is a manifold with positive local reach, Theorems~\ref{thm:Kloc_upper} and~\ref{thm:KU_manifold_limit} establish the variation function of the tangent space as a generalised measure of incoherence.
Assuming that the best approximation is sufficiently incoherent with respect to this measure and that the curvature of the model class is sufficiently small, we can find a neighbourhood of the best approximation for which the variation function is small.
Restricting the optimisation to this neighbourhood ensures that a quasi-optimal empirical approximation can be found with high probability.
We argue that many gradient descent algorithms remain in such a neighbourhood, which provides a heuristic argument for their empirically observed success.

The incoherence and curvature conditions that are required by this argument are discussed in Remark~\ref{rmk:restriction} and investigated for three commonly used model classes in
Examples~\ref{ex:polynomial_regression},~\ref{ex:sparse_regression} and~\ref{ex:matrix_completion}.
These examples illustrate that the conditions
are not trivial to verify and unrealistic to guarantee in most practical settings.
Going forward, we hence advocate for exploring algorithms that automatically remain in a subclass of sufficiently small variation function.

{
    \emergencystretch=3em
    \printbibliography

@article{eigel2018vmc0,
	doi = {10.1515/cmam-2018-0028},
	url = {https://doi.org/10.1515%2Fcmam-2018-0028},
	year = 2018,
	month = {07},
	publisher = {Walter de Gruyter {GmbH}},
	volume = {19},
	number = {1},
	pages = {39--53},
	author = {Martin Eigel and Johannes Neumann and Reinhold Schneider and Sebastian Wolf},
	title = {Non-intrusive Tensor Reconstruction for High-Dimensional Random {PDEs}},
	journal = {Computational Methods in Applied Mathematics}
}

@article{rauhut_2017_iht,
    title = {Low rank tensor recovery via iterative hard thresholding},
    journal = {Linear Algebra and its Applications},
    volume = {523},
    pages = {220--262},
    year = {2017},
	issn = "0024-3795",
	doi = "https://doi.org/10.1016/j.laa.2017.02.028",
	url = "http://www.sciencedirect.com/science/article/pii/S0024379517301295",
    author={Rauhut, Holger and Schneider, Reinhold and Stojanac, {\v{Z}}eljka},
    publisher = {Elsevier}
}

@article{cohen_2017_least-squares,
	author = {Cohen, Albert and Migliorati, Giovanni},
	title = {Optimal weighted least-squares methods},
	year = {2017},
	month = {10},
	journal = {The {SMAI} journal of computational mathematics},
    publisher = {Soci\'et\'e de Math\'ematiques Appliqu\'ees et Industrielles},
	volume = {3},
	pages = {181--203},
	doi = {10.5802/smai-jcm.24},
    mrnumber = {3716755},
    zbl = {1416.62177},
    url = {https://smai-jcm.centre-mersenne.org/articles/10.5802/smai-jcm.24/}
}

@book{hackbusch_2012_book,
    title={Tensor spaces and numerical tensor calculus},
    author={Hackbusch, Wolfgang},
    volume={42},
    year={2012},
    month={2},
    pages={524},
	doi = {10.1007/978-3-642-28027-6},
	url = {https://doi.org/10.1007%2F978-3-642-28027-6},
    publisher={Springer Science \& Business Media}
}

@article{ESTW19,
    IDS = {eigel2019variational,eigel_variational_2019},
	doi = {10.1007/s10444-019-09723-8},
	url = {https://doi.org/10.1007%2Fs10444-019-09723-8},
	year = 2019,
	month = {10},
	publisher = {Springer Science and Business Media {LLC}},
	volume = {45},
	number = {5-6},
	pages = {2503--2532},
    ISSN={1572-9044},
    author={Eigel, Martin and Schneider, Reinhold and Trunschke, Philipp and Wolf, Sebastian},
	title = {Variational {M}onte {C}arlo {\textemdash} bridging concepts of machine learning and high-dimensional partial differential equations},
	journal = {Advances in Computational Mathematics}
}

@incollection{bohn_2018_sparse_grid,
    IDS = {bohn2018convergence},
	title = {On the Convergence Rate of Sparse Grid Least Squares Regression},
    author={Bohn, Bastian},
	doi = {10.1007/978-3-319-75426-0_2},
	year = {2018},
    booktitle = {Sparse Grids and Applications - Miami 2016},
    series = {Lecture Notes in Computational Science and Engineering},
	publisher = {Springer International Publishing},
	pages = {19--41},
    editor = {J. Garcke and D. Pflüger and C. Webster and G. Zhang},
    inspreprintnum={1711},
    volume={123}
}

@article{rauhut_2016_weighted_l1,
	doi = {10.1016/j.acha.2015.02.003},
	url = {https://doi.org/10.1016%2Fj.acha.2015.02.003},
	year = 2016,
	month = {3},
	publisher = {Elsevier {BV}},
	volume = {40},
	number = {2},
	pages = {321--351},
	author = {Holger Rauhut and Rachel Ward},
	title = {Interpolation via weighted $\ell_1$ minimization},
	journal = {Applied and Computational Harmonic Analysis}
}

@article{Candes2010convexRelaxationMatrixCompletion,
	author = {Emmanuel J. Candes and Terence Tao},
	journal = {{IEEE} Transactions on Information Theory},
	publisher = {Institute of Electrical and Electronics Engineers ({IEEE})},
    title={The Power of Convex Relaxation: Near-Optimal Matrix Completion},
    year={2010},
    volume={56},
    number={5},
    pages={2053--2080},
    doi={10.1109/TIT.2010.2044061}
}

@article{candes_2006_stable_recovery,
	author = {Emmanuel J. Cand{\`{e}}s and Justin K. Romberg and Terence Tao},
    title = {Stable signal recovery from incomplete and inaccurate measurements},
    journal = {Communications on Pure and Applied Mathematics},
	publisher = {Wiley},
    volume = {59},
    number = {8},
    pages = {1207--1223},
    doi = {10.1002/cpa.20124},
    year = {2006}
}

@article{yuan2015tensor_completion,
	doi = {10.1007/s10208-015-9269-5},
	url = {https://doi.org/10.1007%2Fs10208-015-9269-5},
	year = 2015,
	month = {6},
	publisher = {Springer Science and Business Media {LLC}},
	volume = {16},
	number = {4},
	pages = {1031--1068},
	author = {Ming Yuan and Cun-Hui Zhang},
	title = {On Tensor Completion via Nuclear Norm Minimization},
	journal = {Foundations of Computational Mathematics}
}

@article{hitchcock1927cp_format,
    author = {Hitchcock, Frank L.},
    title = {The Expression of a Tensor or a Polyadic as a Sum of Products},
    journal = {Journal of Mathematics and Physics},
    volume = {6},
    number = {1-4},
    pages = {164--189},
    doi = {10.1002/sapm192761164},
    url = {https://onlinelibrary.wiley.com/doi/abs/10.1002/sapm192761164},
    year = {1927}
}

@misc{BEST21,
    IDS = {bayer_pricing_2021},
    title={Pricing high-dimensional Bermudan options with hierarchical tensor formats},
	author={Bayer, Christian and Eigel, Martin and Sallandt, Leon and Trunschke, Philipp},
    year={2021},
	month = {03},
    eprint={2103.01934},
    archivePrefix={arXiv},
    primaryClass={q-fin.CP}
}

@article{eigel_2022_convergence,
    doi = {10.1051/m2an/2021070},
    url = {https://doi.org/10.1051/m2an/2021070},
    year = {2022},
    month = jan,
    publisher = {{EDP} Sciences},
    volume = {56},
    number = {1},
    pages = {79--104},
    author = {Martin Eigel and Reinhold Schneider and Philipp Trunschke},
    title = {Convergence bounds for empirical nonlinear least-squares},
    journal = {{ESAIM}: Mathematical Modelling and Numerical Analysis}
}

@article{gels_multidimensional_2019,
	title = {Multidimensional {Approximation} of {Nonlinear} {Dynamical} {Systems}},
	volume = {14},
	issn = {1555-1415, 1555-1423},
	url = {https://asmedigitalcollection.asme.org/computationalnonlinear/article/doi/10.1115/1.4043148/726935/Multidimensional-Approximation-of-Nonlinear},
	doi = {10.1115/1.4043148},
	number = {6},
	urldate = {2019-12-17},
	journal = {Journal of Computational and Nonlinear Dynamics},
	author = {Gelß, Patrick and Klus, Stefan and Eisert, Jens and Schütte, Christof},
	month = jun,
	year = {2019},
}

@inproceedings{goette_2020_dynamical_laws,
	title = {Tensor network approaches for data-driven identiﬁcation of non-linear dynamical laws},
	url = {https://tensorworkshop.github.io/NeurIPS2020/accepted_papers/NIPS_LDL.pdf},
    booktitle = {Proceedings of the 1st Workshop on Quantum Tensor Networks in Machine Learning},
	author = {Goeßmann, Alex and Götte, Michael and Roth, Ingo and Sweke, Ryan and Kutyniok, Gitta and Eisert, Jens},
	month = {12},
	year = {2020},
}

@phdthesis{haberstich_2020_thesis,
    title = {{Adaptive approximation of high-dimensional functions with tree tensor networks for Uncertainty Quantification}},
    author = {Haberstich, Cecile},
    url = {https://tel.archives-ouvertes.fr/tel-03185160},
    number = {2020ECDN0045},
    school = {{{\'E}cole centrale de Nantes}},
    year = {2020},
    month = {12},
    eprint = {tel-03185160},
    eprinttype = {hal},
    eprintclass = {Theses}
}

@article{numpy,
    IDS = {oliphant_guide_2006},
    title         = {Array programming with {NumPy}},
    author        = {Charles R. Harris and K. Jarrod Millman and St{\'{e}}fan J.
                    van der Walt and Ralf Gommers and Pauli Virtanen and David
                    Cournapeau and Eric Wieser and Julian Taylor and Sebastian
                    Berg and Nathaniel J. Smith and Robert Kern and Matti Picus
                    and Stephan Hoyer and Marten H. van Kerkwijk and Matthew
                    Brett and Allan Haldane and Jaime Fern{\'{a}}ndez del
                    R{\'{i}}o and Mark Wiebe and Pearu Peterson and Pierre
                    G{\'{e}}rard-Marchant and Kevin Sheppard and Tyler Reddy and
                    Warren Weckesser and Hameer Abbasi and Christoph Gohlke and
                    Travis E. Oliphant},
    year          = {2020},
    month         = {09},
    journal       = {Nature},
    volume        = {585},
    number        = {7825},
    pages         = {357--362},
    doi           = {10.1038/s41586-020-2649-2},
    publisher     = {Springer Science and Business Media {LLC}},
    url           = {https://doi.org/10.1038/s41586-020-2649-2}
}

@phdthesis{wolf_sebastian_low_2019,
    IDS = {wolf2019},
    title = {Low rank tensor decompositions for high dimensional data approximation, recovery and prediction},
	school = {Technische Universität Berlin},
    author = {Wolf, Alexander Sebastian Johannes Wolf},
	year = {2019},
    doi = {10.14279/DEPOSITONCE-8109},
    url = {https://depositonce.tu-berlin.de/handle/11303/8986},
}

@misc{goessmann2020restricted,
    title={The Restricted Isometry of {ReLU} Networks: Generalization through Norm Concentration}, 
    author={Alex Goeßmann and Gitta Kutyniok},
    year={2020},
    eprint={2007.00479},
    archivePrefix={arXiv},
    primaryClass={stat.ML}
}

@article{recht_2010_nuclear_norm_minimization,
	doi = {10.1137/070697835},
	url = {https://doi.org/10.1137%2F070697835},
	year = 2010,
	month = {01},
	publisher = {Society for Industrial {\&} Applied Mathematics ({SIAM})},
	volume = {52},
	number = {3},
	pages = {471--501},
	author = {Benjamin Recht and Maryam Fazel and Pablo A. Parrilo},
	title = {Guaranteed Minimum-Rank Solutions of Linear Matrix Equations via Nuclear Norm Minimization},
	journal = {{SIAM} Review}
}

@inproceedings{bouchot_2015_CSPG,
	doi = {10.1109/sampta.2015.7148947},
	url = {https://doi.org/10.1109%2Fsampta.2015.7148947},
	year = 2015,
	month = {05},
	publisher = {{IEEE}},
	author = {Jean-Luc Bouchot and Benjamin Bykowski and Holger Rauhut and Christoph Schwab},
	title = {Compressed sensing Petrov-Galerkin approximations for parametric {PDEs}},
	booktitle = {2015 International Conference on Sampling Theory and Applications ({SampTA})}
}

@article{cockreham_2017_reach,
	doi = {10.1007/s00365-017-9388-0},
	url = {https://doi.org/10.1007%2Fs00365-017-9388-0},
	year = 2017,
	month = {08},
	publisher = {Springer Science and Business Media {LLC}},
	volume = {47},
	number = {2},
	pages = {357--371},
	author = {James Cockreham and Fuchang Gao},
	title = {Metric Entropy of Classes of Sets with Positive Reach},
	journal = {Constructive Approximation}
}

@article{seeger_2010_Hausdorff_distance,
    TITLE = {{Distances between closed convex cones: old and new results}},
    AUTHOR = {Iusem, Alfredo and Seeger, Alberto},
    URL = {https://hal.archives-ouvertes.fr/hal-02187073},
    JOURNAL = {{Journal of Convex Analysis}},
    PUBLISHER = {{Heldermann}},
    VOLUME = {17},
    NUMBER = {3-4},
    PAGES = {1033--1055},
    YEAR = {2010},
    eprint = {hal-02187073},
    eprinttype = {hal},
    eprintclass = {Journal articles}
}

@article{scipy,
    author  = {Virtanen, Pauli and Gommers, Ralf and Oliphant, Travis E. and
              Haberland, Matt and Reddy, Tyler and Cournapeau, David and
              Burovski, Evgeni and Peterson, Pearu and Weckesser, Warren and
              Bright, Jonathan and {van der Walt}, St{\'e}fan J. and
              Brett, Matthew and Wilson, Joshua and Millman, K. Jarrod and
              Mayorov, Nikolay and Nelson, Andrew R. J. and Jones, Eric and
              Kern, Robert and Larson, Eric and Carey, C J and
              Polat, {\.I}lhan and Feng, Yu and Moore, Eric W. and
              {VanderPlas}, Jake and Laxalde, Denis and Perktold, Josef and
              Cimrman, Robert and Henriksen, Ian and Quintero, E. A. and
              Harris, Charles R. and Archibald, Anne M. and
              Ribeiro, Ant{\^o}nio H. and Pedregosa, Fabian and
              {van Mulbregt}, Paul and {SciPy 1.0 Contributors}},
    title   = {{{SciPy} 1.0: Fundamental Algorithms for Scientific
              Computing in Python}},
    journal = {Nature Methods},
    year    = {2020},
    volume  = {17},
    pages   = {261--272},
    adsurl  = {https://rdcu.be/b08Wh},
    doi     = {10.1038/s41592-019-0686-2},
}

@article{matplotlib,
  Author    = {Hunter, J. D.},
  Title     = {Matplotlib: A 2D graphics environment},
  Journal   = {Computing in Science \& Engineering},
  Volume    = {9},
  Number    = {3},
  Pages     = {90--95},
  publisher = {IEEE COMPUTER SOC},
  doi       = {10.1109/MCSE.2007.55},
  year      = 2007
}

@article{Grasedyck_2013_survey,
    IDS = {Grasedyck2013},
	doi = {10.1002/gamm.201310004},
	url = {https://doi.org/10.1002%2Fgamm.201310004},
	year = {2013},
	month = {08},
	publisher = {Wiley},
	volume = {36},
	number = {1},
	pages = {53--78},
	author = {Lars Grasedyck and Daniel Kressner and Christine Tobler},
	title = {A literature survey of low-rank tensor approximation techniques},
	journal = {{GAMM}-Mitteilungen}
}

@book{boissonnat2018geometric,
	title = {Geometric and Topological Inference},
    author = {Boissonnat, Jean-Daniel and Chazal, Fr{\'e}d{\'e}ric and Yvinec, Mariette},
    url = {https://hal.inria.fr/hal-01615863},
    note = {Cambridge Texts in Applied Mathematics},
	publisher = {Cambridge University Press},
    year = {2018},
	month = {09},
    doi = {10.1017/9781108297806}
}

@article{Kressner2014,
    author = {Kressner, Daniel and Steinlechner, Michael and Vandereycken, Bart},
    title = {Low-rank tensor completion by {R}iemannian optimization},
    journal = {BIT Numerical Mathematics},
    volume = {54},
    year = {2014},
    number = {2},
    pages = {447--468},
    issn = {0006-3835},
    mrclass = {15A69 (58C05 65F99 65K05)},
    mrnumber = {3223510},
    doi = {10.1007/s10543-013-0455-z}
}

@misc{cohen_2021_lognormal,
  doi = {10.48550/ARXIV.2103.13935},
  url = {https://arxiv.org/abs/2103.13935},
  author = {Cohen, Albert and Migliorati, Giovanni},
  title = {Near-optimal approximation methods for elliptic PDEs with lognormal coefficients},
  publisher = {arXiv},
  year = {2021},
  copyright = {Creative Commons Attribution Non Commercial No Derivatives 4.0 International}
}

@article{federer_1959_curvature_measures,
  doi = {10.1090/s0002-9947-1959-0110078-1},
  url = {https://doi.org/10.1090/s0002-9947-1959-0110078-1},
  year = {1959},
  publisher = {American Mathematical Society ({AMS})},
  volume = {93},
  number = {3},
  pages = {418--491},
  author = {Herbert Federer},
  title = {Curvature measures},
  journal = {Transactions of the American Mathematical Society}
}

@article{cuevas_2012_rolling,
  title={On Statistical Properties of Sets Fulfilling Rolling-Type Conditions},
  volume={44},
  DOI={10.1239/aap/1339878713},
  number={2},
  journal={Advances in Applied Probability},
  publisher={Cambridge University Press},
  author={Cuevas, A. and Fraiman, R. and Pateiro-López, B.},
  year={2012},
  pages={311–329}
}

@article{Montgomery-Smith_1990_Rademacher_sums,
  doi = {10.1090/s0002-9939-1990-1013975-0},
  url = {https://doi.org/10.1090/s0002-9939-1990-1013975-0},
  year = {1990},
  publisher = {American Mathematical Society ({AMS})},
  volume = {109},
  number = {2},
  pages = {517--522},
  author = {S. J. Montgomery-Smith},
  title = {The distribution of Rademacher sums},
  journal = {Proceedings of the American Mathematical Society}
}

@misc{kraemer2015ttpde,
  doi = {10.48550/ARXIV.1509.00311},
  author = {Grasedyck, Lars and Kluge, Melanie and Krämer, Sebastian},
  title = {Alternating Least Squares Tensor Completion in The TT-Format},
  publisher = {arXiv},
  year = {2015},
  copyright = {arXiv.org perpetual, non-exclusive license}
}
}

\section*{Acknowledgements}
Our code made use of the following Python packages: 
\texttt{numpy}~\cite{numpy}, \texttt{scipy}~\cite{scipy}, and \texttt{matplotlib}~\cite{matplotlib}.

\appendix
\section{Appendix: Proof of Theorem~\ref{thm:K_properties}} \label{proof:thm:K_properties}

\begin{enumerate}
    \item Follows directly from the definition.
    
    \item To see that $\mfrak{K}_{A} = \mfrak{K}_{\cl\pars{A}}$ let $a\in\cl\pars{A}\setminus\braces{0}$.
    Then there exists a sequence $\braces{a_k} \in A\setminus\braces{0}$ such that $a_k \to a$.
    Due to the continuity of $a \mapsto a\pars{y}^2/\norm{\bullet}^2$ on $A\setminus\braces{0}$ it follows that
    \begin{equation}
        \mfrak{K}_{\braces{a}}\pars{y}
        = \frac{\abs{a\pars{y}}^2}{\norm{a}^2}
        = \lim_{k\to\infty} \frac{\abs{a_k\pars{y}}^2}{\norm{a_k}^2}
        = \lim_{k\to\infty} \mfrak{K}_{\braces{a_k}}\pars{y}
    \end{equation}
    And since $\mfrak{K}_{\braces{a_k}} \le \mfrak{K}_{A}$ for all $k=1,\ldots,\infty$ and we can conclude $\mfrak{K}_{\braces{a}} \le \mfrak{K}_{A}$.
    The assertion follows with \ref{thm:K_properties:union} since $\mfrak{K}_{A} \le \mfrak{K}_{\cl\pars{A}} = \sup_{a\in\cl\pars{A}}\mfrak{K}_{\braces{a}} \le \mfrak{K}_{A}$.

    \item[3.-5.] In all three case we can write $\mfrak{K}_{\bullet} = \operatorname{sq}\circ \sup\circ \operatorname{abs}\circ\, U$ with
    \begin{align}
        \operatorname{sq} &: \mcal{V}_{w,\infty} \to \mcal{V}_{w^2,\infty},
        & \operatorname{sq}\pars{v}\pars{y} &:= v\pars{y}^2, \label{eq:sq} \\
        \operatorname{abs} &: \mcal{V}_{w,\infty} \to \mcal{V}_{w,\infty},
        & \operatorname{abs}\pars{v}\pars{y} &:= \abs{v\pars{y}}, \label{eq:abs} \\
        \sup &: \mfrak{P}\pars{\mcal{V}_{w,\infty}} \to \mcal{V}_{w,\infty},
        & \sup\pars{V}\pars{y} &:= \sup_{v\in V} v\pars{y}, \text{ and} \label{eq:sup} \\
        \inf &: \mfrak{P}\pars{\mcal{V}_{w,\infty}} \to \mcal{V}_{w,\infty},
        & \inf\pars{V}\pars{y} &:= \inf_{v\in V} v\pars{y} .
    \end{align}
    This allows us to prove the continuity of $\operatorname{sq}\circ \sup\circ \operatorname{abs} : \mfrak{P}\pars{\mcal{V}_{w,\infty}} \to \mcal{V}_{w^2,\infty}$ and $U$ individually.
    The main difference between Theorems~\ref{thm:K_properties:continuity:compact_sets}, \ref{thm:K_properties:continuity:all_sets} and~\ref{thm:K_properties:continuity:cones} then comes from the domain of $U$.
    
    We proceed by showing that $\operatorname{sq}\circ \sup\circ \operatorname{abs} : \mfrak{P}\pars{\mcal{V}_{w,\infty}} \to \mcal{V}_{w^2,\infty}$ is continuous with respect to the Hausdorff pseudometric, which also implies the continuity of $\operatorname{sq}\circ \sup\circ \operatorname{abs} : \mfrak{C}\pars{\mcal{V}_{w,\infty}} \to \mcal{V}_{w^2,\infty}$ with respect to the Hausdorff metric.
    
    To do this we require the following four lemmata.
    \begin{lemma} \label{lem:continuity_implies_pseudo-Hausdorff_continuity}
        Let $\pars{M_1,d_1}$ and $\pars{M_2,d_2}$ be metric spaces, let $f:M_1\to M_2$ and define $f\pars{X} := \braces{ f\pars{x} : x\in X}$ for any $X\in M_1$.
        If $f$ is uniformly continuous, then $f : \mfrak{P}\pars{M_1}\to\mfrak{P}\pars{M_2}$ is uniformly continuous with respect to the Hausdorff pseudometric.
    \end{lemma}
    \begin{proof}
        Recall, that $d_{\mathrm{H}}\pars{X,Y} \le \varepsilon$ means that
        \begin{equation}
            \forall x\in X\,\exists y\in Y:\, d\pars{x,y} \le \varepsilon
            \qquad\text{and}\qquad
            \forall y\in Y\,\exists x\in X:\, d\pars{y,x} \le \varepsilon .
        \end{equation} 
        Let $\varepsilon>0$.
        Since $f$ is uniformly continuous there exists $\delta>0$ such that $d_1\pars{x,y}<\delta$ implies $d_2\pars{f\pars{x},f\pars{y}}<\varepsilon$.
        We now show that $d_{\mathrm{H}}\pars{U,V}<\delta$ implies $d_{\mathrm{H}}\pars{f\pars{U}, f\pars{V}}<\varepsilon$.

        For this let $f_u\in f\pars{U}$ and choose $u\in U$ such that $f\pars{u} = f_u$.
        Since $d_{\mathrm{H}}\pars{U,V}<\delta$ there exists $v\in V$ such that $d_1\pars{u,v}<\delta$ and consequently $d_2\pars{f\pars{u},f\pars{v}}<\varepsilon$, by uniform continuity.
        This means that for every $f_u\in f\pars{U}$ there exists $f_v \in f\pars{V}$ such that $d_2\pars{f_u, f_v} < \varepsilon$.
        Since this argument remains valid if the roles of $U$ and $V$ are reversed we can conclude that $d_{\mathrm{H}}\pars{f\pars{U}, f\pars{V}} < \varepsilon$.
    \end{proof}
    
    \begin{lemma} \label{lem:abs_continuous}
        $\operatorname{abs} : \mcal{V}_{w,\infty} \to \mcal{V}_{w,\infty}$ is Lipschitz continuous with constant $1$.
        $\operatorname{abs} : \mfrak{P}\pars{\mcal{V}_{w,\infty}} \to \mfrak{P}\pars{\mcal{V}_{w,\infty}}$ is uniformly continuous with respect to the Hausdorff pseudometric.
    \end{lemma}
    \begin{proof}
        The first assertion follows by the reverse triangle inequality, $\abs{\abs{v\pars{y}}-\abs{w\pars{y}}} \le \abs{v\pars{y}-w\pars{y}}$.
        The second asserion follows by Lemma~\ref{lem:continuity_implies_pseudo-Hausdorff_continuity}, since Lipschitz continuity implies uniform continuity.
    \end{proof}
    
    \begin{lemma} \label{lem:sup_continuous}
        $\sup : \mfrak{P}\pars{\mcal{V}_{w,\infty}} \to \mcal{V}_{w,\infty}$ is Lipschitz continuous with constant $1$.
    \end{lemma}
    \begin{proof}
        Let $U,V \in \mfrak{P}\pars{\mcal{V}_{w,\infty}}$.
        To show that $\norm{\sup\pars{U} - \sup\pars{V}}_{w,\infty} \le d_{\mathrm{H}}\pars{U,V}$, fix $y\in Y$ and assume w.l.o.g.\ that $\sup\pars{U}\pars{y} \ge\sup\pars{V}\pars{y}$.
        Then
        \begin{align}
            \sqrt{w\pars{y}} \abs*{\sup\pars{U}\pars{y} - \sup\pars{V}\pars{y}}
            &= \sqrt{w\pars{y}} \sup_{u\in U}\inf_{v\in V} \pars{u\pars{y}-v\pars{y}} \\
            &\le \sup_{u\in U}\inf_{v\in V} \norm{u-v}_{w,\infty} \\
            &\le d_{\mathrm{H}}\pars{U, V} .
        \end{align}
        Taking the supremum over $y\in Y$ proves the assertion.
    \end{proof}
    
    \begin{lemma} \label{lem:sq_continuous}
        $\operatorname{sq} : \mcal{V}_{w,\infty} \to \mcal{V}_{w^2,\infty}$ is continuous.
    \end{lemma}
    \begin{proof}
        Fix $v\in\mcal{V}_{w,\infty}$ and let $w\in\mcal{V}_{w,\infty}$ be arbitrary.
        Then 
        \begin{align}
            \norm{v^2 - w^2}_{w^2,\infty}
            &= \norm{vv - vw + vw - ww}_{w^2,\infty} \\
            &\le \norm{v\pars{v-w}}_{w^2,\infty} + \norm{w\pars{v-w}}_{w^2,\infty} \\
            &\le \pars{\norm{v}_{w,\infty}+\norm{w}_{w,\infty}}\norm{v-w}_{w,\infty} .
        \end{align}
        Observe that, due the reverse triangle inequality, $\norm{v-w}_{w,\infty}\le\delta$ implies $\norm{w}_{w,\infty} \le \norm{v}_{w,\infty}+\delta$.
        This proves continuity, since for any $\varepsilon$ we can choose $\delta$ such that $\norm{v-w}_{w,\infty}\le \delta$ implies
        \begin{equation}
            \norm{v^2 - w^2}_{w^2,\infty} \le \pars{2\norm{v}_{w,\infty} + \delta} \delta \le \varepsilon .
        \end{equation}
    \end{proof}
    
    As a composition of continuous functions, the continuity of $\operatorname{sq}\circ \sup\circ \operatorname{abs} : \mfrak{P}\pars{\mcal{V}_{w,\infty}} \to \mcal{V}_{w^2,\infty}$ is guaranteed by Lemmas~\ref{lem:abs_continuous} to \ref{lem:sq_continuous}.
    We now proceed by proving the continuity of $U$ in the settings of Theorem~\ref{thm:K_properties:continuity:compact_sets}, \ref{thm:K_properties:continuity:all_sets} and~\ref{thm:K_properties:continuity:cones} individually.
     
    \item To prove this we need the subsequent lemma.
    \begin{lemma} \label{lem:continuity_implies_Hausdorff_continuity}
        Let $\pars{M_1,d_1}$ and $\pars{M_2,d_2}$ be metric spaces, let $f:M_1\to M_2$ and define $f\pars{X} := \braces{ f\pars{x} : x\in X}$ for any $X\in M_1$.
        If $f$ is continuous, then $f : \mfrak{C}\pars{M_1}\to\mfrak{C}\pars{M_2}$ is continuous with respect to the Hausdorff metric.
    \end{lemma}
    \begin{proof}
        $f : \mfrak{C}\pars{M_1}\to\mfrak{C}\pars{M_2}$ is well-defined since the image of a compact set under a continuous function is compact.
        Now recall, that $d_{\mathrm{H}}\pars{X,Y} \le \varepsilon$ means that
        \begin{equation}
            \forall x\in X\,\exists y\in Y:\, d\pars{x,y} \le \varepsilon
            \qquad\text{and}\qquad
            \forall y\in Y\,\exists x\in X:\, d\pars{y,x} \le \varepsilon .
        \end{equation}
        
        Let $\varepsilon>0$ and $U\in\mfrak{C}\pars{M_1}$.
        Since $f$ is continuous in every $u\in U$ there exists a $\delta_u > 0$ that guarantees
        \begin{equation}
            d_1\pars{u,\tilde{u}} \le \delta_u \Rightarrow d_2\pars{f\pars{u}, f\pars{\tilde{u}}} \le \frac{\varepsilon}{2} .
        \end{equation}
        Now define the sets $N_u := \braces{\tilde{u}\in M_1 : d_1\pars{u,\tilde{u}} \le \frac{\delta_u}{2}}$.
        Since $u\in N_u$, the family $\braces{N_u}_{u\in U}$ defines a covering of $U$ and since $U$ is compact there exists a finite subcovering $\braces{N_{u_i}}_{i=1,\ldots,n}$.
        Choose $\delta := \min_{i=1,\ldots,n}\frac{\delta_{u_i}}{2}$ and note, that $\delta$ has to be positive, since it is the minimum of finitely many positive numbers.
        Now let $V\in\mfrak{C}\pars{M_1}$ such that $d_{\mathrm{H}}\pars{U,V}\le\delta$.
        
        First, we show that
        \begin{equation}
            \forall f_v\in f\pars{V}\exists f_u\in f\pars{U} : d_2\pars{f_v, f_u} \le \varepsilon .
        \end{equation}
        For this let $v\in V$ be any element that satisfies $f\pars{v} = f_v$.
        Since $d_{\mathrm{H}}\pars{U,V}\le\delta$ there exists $u\in U$ with $d_1\pars{u,v}\le\delta$.
        Moreover, by definition of the covering $\braces{N_{u_i}}_{i=1,\ldots,n}$, there exists $u_i$ such that $d_1\pars{u,u_i}\le\frac{\delta_{u_i}}{2}$.
        Using the triangle inequality, we thus obtain
        \begin{equation}
            d_1\pars{u_i,v} \le d_1\pars{u_i,u} + d_1\pars{u,v} \le \frac{\delta_{u_i}}{2} + \delta \le \delta_{u_i}
        \end{equation}
        and the definition of $\delta_{u_i}$ finally yields $d_2\pars{f\pars{v}, f\pars{u_i}} \le \frac{\varepsilon}{2} \le \varepsilon$.
        
        Now we show that
        \begin{equation}
            \forall f_u\in f\pars{U}\exists f_v\in f\pars{V} : d_2\pars{f_u, f_v} \le \varepsilon .
        \end{equation}
        Analogously to the argument from above let $u\in U$ be any element that satisfies $f\pars{u}=f_u$.
        Since $d_{\mathrm{H}}\pars{U,V}\le\delta$ there exists $v\in V$ with $d_1\pars{u,v}\le\delta$ and by the definition of the covering there exists also a $u_i$ with $d_1\pars{u,u_i}\le\frac{\delta_{u_i}}{2}$.
        We can now estimate
        \begin{equation}
            d_2\pars{f\pars{u},f\pars{v}} \le d_2\pars{f\pars{u},f\pars{u_i}} + d_2\pars{f\pars{u_i},f\pars{v}} \le \frac{\varepsilon}{2} + \frac{\varepsilon}{2} = \varepsilon
        \end{equation}
        which holds by the definition of $\delta_{u_i}$ and because
        \begin{equation}
            d_1\pars{u_i,v} \le d_1\pars{u_i,u} + d_1\pars{u,v} \le \frac{\delta_{u_i}}{2} + \delta \le \delta_{u_i} .
        \end{equation}
    \end{proof}
    
    Since the function $u \mapsto u/\norm{u}$ is continuous on $\mcal{V}_{w,\infty}\setminus \braces{0}$ the function $U : \mfrak{C}\pars{\mcal{V}_{w,\infty}\setminus\braces{0}} \to \mfrak{C}\pars{S\pars{0,1}\cap\mcal{V}_{w,\infty}}$ is continuous by Lemma~\ref{lem:continuity_implies_Hausdorff_continuity}.
    
    \item Let $r>0$.
    Since the function $u \mapsto u/\norm{u}$ is uniformly continuous on $\mcal{V}_{w,\infty}\setminus B\pars{0,r}$ the function $U : \mfrak{P}\pars{\mcal{V}_{w,\infty}\setminus B\pars{0,r}} \to \mfrak{P}\pars{S\pars{0,1}\cap\mcal{V}_{w,\infty}}$ is uniformly continuous by Lemma~\ref{lem:continuity_implies_pseudo-Hausdorff_continuity}.
    
    \item By definition of the truncated Hausdorff distance, $U : \operatorname{Cone}\pars{\mfrak{P}\pars{\mcal{V}_{w,\infty}}} \to \mfrak{P}\pars{S\pars{0,1}\cap\mcal{V}_{w,\infty}}$ is Lipschitz continuous with constant $1$.
    
    \item Every $v\in A + B$ can be written as $v = \vec{v}^\intercal\alpha$ for some $\alpha\in\mathbb{R}^2$ and $\vec{v}\in \pars{A\times B}\setminus\braces{0}$.
    Moreover, $A\perp B$ implies that $\norm{v}^2 = \alpha^\intercal D(\vec{v})^2 \alpha$ with $D(\vec{v}) := \operatorname{diag}\pars{\norm{\vec{v}_1}, \norm{\vec{v}_2}}$.
    Now define $C_{\vec{v},y} = D(\vec{v})^{-1}\vec{v}(y)$ and observe that
    \begin{align}
        \mfrak{K}_{A + B}\pars{y} 
        &\le \sup_{\vec{v}\in \pars{A\times B}\setminus\braces{0}} \sup_{\alpha\in\mathbb{R}^2\setminus\braces{0}}\frac{\abs{\alpha^\intercal \vec{v}(y)\vec{v}(y)^\intercal \alpha}}{\alpha^\intercal D(\vec{v})^2 \alpha} \\
        &= \sup_{\vec{v}\in \pars{A\times B}\setminus\braces{0}} \sup_{\beta\in\mathbb{R}^2\setminus\braces{0}} \frac{\abs{\beta^\intercal C_{\vec{v},y}C_{\vec{v},y}^\intercal \beta}}{\beta^\intercal \beta} \\
        &= \sup_{\vec{v}\in \pars{A\times B}\setminus\braces{0}} \norm{C_{\vec{v},y}}_2^2 \\
        &= \sup_{\vec{v}_1\in A\setminus\braces{0}} \sup_{\vec{v}_2\in B\setminus\braces{0}} \tfrac{\abs{\vec{v}_1(y)}^2}{\norm{\vec{v}_1}^2} + \tfrac{\abs{\vec{v}_2(y)}^2}{\norm{\vec{v}_2}^2} \\
        &= \mfrak{K}_{A}\pars{y} + \mfrak{K}_{B}\pars{y} .
    \end{align}
    Note that the first inequality is indeed an equality, if $A$ and $B$ are linear spaces.

    \item Let $a\in A$ and $b\in B$.
    Since $a\indep b$ also $a^2\indep b^2$ and consequently $\norm{a\cdot b}^2 = \mbb{E}\bracs{a^2 b^2} = \mbb{E}\bracs{a^2} \mbb{E}\bracs{b^2} = \norm{a}^2\norm{b}^2$.
    Now recall that $\mathfrak{K}_A\pars{y} = \sup_{a\in U\pars{A}} a\pars{y}^2$. Thus
    \begin{equation}
        \mathfrak{K}_{A\cdot B}\pars{y}
        = \sup_{a\in A} \sup_{b\in B} \frac{\pars{a\cdot b}\pars{y}^2}{\norm{a\cdot b}^2}
        = \sup_{a\in A} \sup_{b\in B} \frac{a\pars{y}^2\cdot b\pars{y}^2}{\norm{a}^2\norm{b}^2}
        = \mathfrak{K}_{A}\pars{y} \cdot \mathfrak{K}_{B}\pars{y} .
    \end{equation}

    \item A direct consequence of Theorem~\ref{thm:K_properties:sum} is the following lemma.
    \begin{lemma} \label{lem:K_properties:span}
        Let $\braces{P_j}_{j\in J}$ be an orthonormal basis for $A$. Then $\mfrak{K}_{A}\pars{y} = \sum_{j\in J} P_j\pars{y}^2$.
    \end{lemma}
    \begin{proof}
        Trivial.
    \end{proof}
    Now let $\braces{P_{A,j}}_{j\in J}$ be an orthonormal basis of $A$ and $\braces{P_{B,k}}_{k\in K}$ be an orthonormal basis of $B$.
    Then $\braces{P_{A,j}\otimes P_{B,k}}_{j\in J, k\in K}$ is an orthonormal basis for $A\otimes B$ and by Lemma~\ref{lem:K_properties:span}
    \begin{align}
        \mfrak{K}_{A\otimes B}\pars{y}
        &= \sum_{j\in J}\sum_{k\in K} P_{A,j}\pars{y}^2\cdot P_{B,k}\pars{y}^2 \\
        &= \pars*{\sum_{j\in J} P_{A,j}\pars{y}^2}\cdot\pars*{\sum_{k\in K} P_{B,k}\pars{y}^2} \\
        &= \mfrak{K}_A\pars{y}\cdot \mfrak{K}_B\pars{y} .
    \end{align}
\end{enumerate}

\section{Appendix: Proof of Proposition~\ref{prop:nhood_properties}} \label{proof:prop:nhood_properties}

The first statement indeed characterizes the reach of a set and an easily accessible proof can be found in \cite[Theorem~7.8~(2)]{boissonnat2018geometric}.
We reiterate this proof in the following, since the proof of the second statement relies on similar arguments.

\begin{enumerate}
\item

Let $v\in\mcal{N}\cap B\pars{u,r}$.
Then there exists a unique best approximation of $v$ in $u+\mbb{T}_u\mcal{N}$ which we denote by $w$.
By the Pythagorean theorem, it holds that $\norm{u-w}^2 = \norm{u-v}^2 - \norm{v-w}^2 \le r^2$ and thus $w\in\pars{u+\mbb{T}_u\mcal{N}}\cap B\pars{u,r}$.

To show that $\norm{v-w}\le C\norm{u-v}^2$, we consider the intersection of the sets $\mcal{N}$ and $u+\mbb{T}_u\mcal{N}$ with the plane spanned by $u$, $v$ and $w$.
Since all three points lie in this plane their relative distances are preserved and it suffices to consider this two-dimensional problem from here on.
Let $D$ be the disk of radius $R$ that is tangent to $\mbb{T}_u\mcal{N}$ at $u$ and whose center $c$ is on the same side of $u+\mbb{T}_u\mcal{N}$ as $v$.
This is illustrated in Figure~\ref{fig:projecting_N_to_TN}.

Because $\norm{v-w} = \norm{u-v}\sin\pars{\alpha}$, it suffices to bound $\sin\pars{\alpha}$.
To this end, observe that $\rch\pars{\mcal{N}, u} = R$ implies that the disc $D$ intersects $\mcal{N}$ only in $u$.
This means, that $v$ does not lie in the interior of $D$ and the line segment $\overline{uv} = \braces{\lambda u + \pars{1-\lambda}v : \lambda\in\bracs{0,1}}$ must intersect the boundary of $D$ at a point $x$.
Since the triangle $\Delta\pars{u,c,x}$ is isosceles, we have $\beta = 2\alpha$ and $\norm{u-x} = 2R\sin\pars{\frac{\beta}{2}} = 2R\sin\pars{\alpha}$.
Using $\norm{u-x}\le\norm{u-v}$ finally yields
\begin{equation} \label{eq:v-w_bound}
	\norm{v-w} = \norm{u-v}\sin\pars{\alpha} = \norm{u-v}\frac{\norm{u-x}}{2R} \le \frac{\norm{u-v}^2}{2R} .
\end{equation}

\begin{figure}[H]
    \centering
    \newcommand{\degre}{\ensuremath{^\circ}}
    \definecolor{angleColor}{rgb}{0.5,0.5,0.5}
    \definecolor{denim}{rgb}{0.08, 0.38, 0.74}
    \tikzset{
    dot/.style = {circle, fill, minimum size=#1, inner sep=0pt, outer sep=0pt},
    dot/.default = 3.5pt  
    }
    \newcommand*\curvature{0.1}
    \newcommand*\radius{3.0}  
    \newcommand*\vx{5.5}
    \pgfmathdeclarefunction{Parabola}{1}{\pgfmathparse{-(#1)^2*\curvature/2}}
    \begin{tikzpicture}[line cap=round,line join=round, scale=0.865]
    \newcommand*\xlim{7.5}
    \clip(-\xlim,-4) rectangle (\xlim,1);  
    
    \coordinate (u) at (0,0);
    \coordinate (c) at (0,-\radius);
    \coordinate (v) at ($(\vx,{Parabola(\vx)})$);
    \coordinate (w) at (\vx,0);
    
    \draw [line width=1.2pt,domain=-6.5:6.5,samples=50,name path=parabola] plot (\x, {Parabola(\x)});
    {
        \def\Nx{-5.75}
        \draw[color=black] ($(\Nx,{Parabola(\Nx)})$) node[anchor=north] {$\mcal{N}$};
    }
    \draw [line width=1.2pt,domain=-\xlim:\xlim,name path=line] plot (\x, {0*\x});
    {
        \def\TuNx{-6.5}
        \draw[color=black] (\TuNx,0) node[anchor=south] {$u+\mbb{T}_u\mcal{N}$};
    }
    \draw [name path=circle] (c) circle (\radius);
    {
        \def\Dphi{170}
        \draw[color=black] ($(c) + (\Dphi:\radius)$) node[anchor=east] {$D$};
    }
    
    \node[dot, label=above:$u$] at (u) {};
    \node[dot, label=below:$v$] at (v) {};
    \node[dot, label=above:$w$] at (w) {};
    \draw (v) -- (w);
    \draw [dash pattern=on 3pt off 3pt,name path=segment] (u) -- (v);
    \path [name intersections={of = circle and segment}];
	\node[dot, label={[shift={(-0.4,-0.4)}]:$x$}] (x) at (intersection-1) {};
    \node[dot, label=below:$c$] at (c) {};
    \draw (c) -- (u) node [midway, left] {$R$}
          (c) -- (x); 
          
    \draw let
        \p1 = (x),
        \p2 = (u),
        \n1 = {veclen(\x1-\x2,\y1-\y2)},
    in
        \pgfextra{\pgfmathsetmacro{\sinAngle}{\n1/(2cm*\radius)}}
        \pgfextra{\pgfmathsetmacro{\startAngle}{90-2*asin(\sinAngle)}}
        ($(c)+(\startAngle:1)$) arc (\startAngle:90:1) node [midway, label={[shift={(-0.03,-0.65)}]$\beta$}] {};
    
    
    \draw let
        \p1 = (u),
        \p2 = (v),
        \p3 = (w),
        \n1 = {veclen(\x1-\x2,\y1-\y2)},  
        \n2 = {veclen(\x2-\x3,\y2-\y3)},  
    in
        \pgfextra{\pgfmathsetmacro{\sinAngle}{\n2/\n1}}
        \pgfextra{\pgfmathsetmacro{\startAngle}{-asin(\sinAngle)}}
        ($(u)+(\startAngle:2)$) arc (\startAngle:0:2) node[label={[shift={(0,-0.07)}]$\alpha$}] {};
    \end{tikzpicture}
    \caption{}
    \label{fig:projecting_N_to_TN}
\end{figure}

\item
Let $P : \mcal{N} \to {u + \mbb{T}_u\mcal{N}}$ be the Euclidean projection from $\mcal{N}$ onto the tangent space of $\mcal{N}$ at $u$ and define 
\begin{equation}
    r^\ast := \min\braces{\sup\braces{r > 0 \,:\, \pars{u+\mbb{T}_u\mcal{N}}\cap B\pars{u,r} \subseteq P\mcal{N}}, R} .
\end{equation}
Since $\mcal{N}$ is a neighbourhood of $u$, this constant is positive and we can consider some fixed $r\in\pars{0, r^\ast}$.
By definition of $r^\ast$, there exists for evey $w\in \pars{u+\mbb{T}_u\mcal{N}}\cap B\pars{u,r}$ an element $\tilde{v}\in\mcal{N}$ such that $P_{\mbb{T}_u\mcal{N}} \tilde{v} = w$.
Analogously to the proof of Proposition~\ref{prop:nhood_properties:tangent_space_projection}, we consider the intersection of the sets $\mcal{N}$ and $u+\mbb{T}_u\mcal{N}$ with the plane spanned by $u$, $\tilde{v}$ and $w$.
Because the distance between these points is preserved by the intersection, we can again consider the resulting two-dimensional problem.
Let $D$ be the disk of radius $R$ that is tangent to $\mbb{T}_u\mcal{N}$ at $u$ and whose center $c$ is on the same side of $u+\mbb{T}_u\mcal{N}$ as $\tilde{v}$.
Since $r< R$, the neighbourhood $\mcal{N}$ is a simply connected manifold and the line segment $\overline{cw} = \braces{\lambda c + \pars{1-\lambda}w : \lambda\in\bracs{0,1}}$ has to intersect $\mcal{N}$ at a point $v$.
Finally, let $x := R\frac{w-c}{\norm{w-c}} + c$ be the intersection of $\overline{cw}$ with the disc $D$.
This is illustrated in Figure~\ref{fig:projecting_TN_to_N}.

Since $\norm{w-v} \le \norm{w-x}$, it suffices to bound $\norm{w-x}$ which is given by the Pythagorean theorem as $\norm{w-x} = \sqrt{R^2+\norm{w-u}^2} - R$.
Defining $\ell\pars{r} := \sqrt{R^2+r^2} - R$ and $\tilde{\ell}\pars{r} := \frac{r^2}{2R}$ we observe that $\ell\pars{r} \le \tilde{\ell}\pars{r}$ since $\ell\pars{0} = 0 = \tilde{\ell}\pars{0}$ and
\begin{equation}
    \ell'\pars{r} = \frac{r}{\sqrt{R^2+r^2}} \le \frac{r}{R} = \tilde{\ell}'\pars{r} .
\end{equation}
This yields $\norm{w-v}\le\norm{w-x} = \ell\pars{\norm{w-u}} \le \tilde\ell\pars{\norm{w-u}} = \frac{r^2}{2R}$ and concludes the proof.



\begin{figure}[H]
    \centering
    \newcommand{\degre}{\ensuremath{^\circ}}
    \definecolor{angleColor}{rgb}{0.5,0.5,0.5}
    \definecolor{denim}{rgb}{0.08, 0.38, 0.74}
    \tikzset{
    dot/.style = {circle, fill, minimum size=#1, inner sep=0pt, outer sep=0pt},
    dot/.default = 3.5pt  
    }
    \newcommand*\curvature{0.1}

    \newcommand*\radius{3.0}  
    \newcommand*\wx{5.5}

    \pgfmathdeclarefunction{Parabola}{1}{\pgfmathparse{-(#1)^2*\curvature/2}}
    \begin{tikzpicture}[line cap=round,line join=round, scale=0.865]
    \newcommand*\xlim{7.5}
    \clip(-\xlim,-4) rectangle (\xlim,1);  
    
    \coordinate (u) at (0,0);
    \coordinate (c) at (0,-\radius);
    \coordinate (w) at (\wx,0);
    
    \draw [line width=1.2pt,domain=-6.5:6.5,samples=50,name path=parabola] plot (\x, {Parabola(\x)});
    {
        \def\Nx{-5.75}
        \draw[color=black] ($(\Nx,{Parabola(\Nx)})$) node[anchor=north] {$\mcal{N}$};
    }
    \draw [line width=1.2pt,domain=-\xlim:\xlim,name path=line] plot (\x, {0*\x});
    {
        \def\TuNx{-6.5}
        \draw[color=black] (\TuNx,0) node[anchor=south] {$u+\mbb{T}_u\mcal{N}$};
    }
    \draw [name path=circle] (c) circle (\radius);
    {
        \def\Dphi{170}
        \draw[color=black] ($(c) + (\Dphi:\radius)$) node[anchor=east] {$D$};
    }
    
    \node[dot, label=above:$u$] at (u) {};
    \node[dot, label=below:$c$] at (c) {};
    \node[dot, label=above right:$w$] at (w) {};
    \draw (c) -- (u) node [midway, left] {$R$};

    \coordinate (tilde_v) at ($(\wx,{Parabola(\wx)})$);
    \node[dot, label=above right:$\tilde{v}$] at (tilde_v) {};
    \draw [dash pattern=on 3pt off 3pt] (tilde_v) -- (w);


    \draw [name path=segment] (c) -- (w);
    \path [name intersections={of = circle and segment}];
    \node[dot, label={[shift={(0.08,0.08)}]$x$}] (x) at (intersection-1) {};
    \path [name intersections={of = parabola and segment}];
    \node[dot, label=above:$v$] (v) at (intersection-1) {};
	
    \end{tikzpicture}
    \caption{}
    \label{fig:projecting_TN_to_N}
\end{figure}
\end{enumerate}

\section{Appendix: Proof of Corollary~\ref{cor:convergence_M} and Proposition~\ref{prop:convergence_UM}} \label{proof:prop:convergence_UM}

\subsection*{Proof of Corollary~\ref{cor:convergence_M}}
Also recall that $d_{\mathrm{H}}\pars{\mcal{M}\cap B\pars{u,r}, \pars{u+\mbb{T}_u\mcal{M}}\cap B\pars{u,r}} \le \frac{r^2}{2R}$ is equivalent to the conjunction of the following two statements.
\begin{enumerate}
    \item For every $v\in \mcal{M}\cap B\pars{u,r}$ there exists a $w\in \pars{u+\mbb{T}_u\mcal{M}}\cap B\pars{u,r}$ such that $\norm{v-w}\le \frac{r^2}{2R}$.
    \item For every $w\in \pars{u+\mbb{T}_u\mcal{M}}\cap B\pars{u,r}$ there exists a $v\in \mcal{M}\cap B\pars{u,r}$ such that $\norm{v-w}\le \frac{r^2}{2R}$.
\end{enumerate}
The statement now follows from Proposition~\ref{prop:nhood_properties}.

\subsection*{Proof of Proposition~\ref{prop:convergence_UM}}
Recall that $R = \operatorname{rch}\pars{\mcal{M}\cap B\pars{u,r_0}}$ and $r \le \min\braces{r_0, R}$ and define $C := \pars{2R}^{-1}$.
To prove $d_{\mathrm{H}}\pars{U\pars{\mcal{M}\cap B\pars{u,r}-u}, U\pars{\mbb{T}_u\mcal{M}}} \le 2Cr$, note that $d_{\mathrm{H}}$ is induced by a norm and is therefore absolutely homogeneous and translation invariant.
Therefore,
\begin{equation}
    d_{\mathrm{H}}\pars{U\pars{\mcal{M}\cap B\pars{u,r}-u}, U\pars{\mbb{T}_u\mcal{M}}} = \frac{1}{r} d_{\mathrm{H}}\pars{rU\pars{\mcal{M}\cap B\pars{u,r}-u}, rU\pars{\mbb{T}_u\mcal{M}}} .
\end{equation}
Now define the operator $U_r\pars{X} := rU\pars{X}$ that scales every element of a set to norm $r$.
The claim follows if $d_{\mathrm{H}}\pars{U_r\pars{\mcal{M}\cap B\pars{u,r}-u}, U_r\pars{\mbb{T}_u\mcal{M}}} \le 2Cr^2$.
To prove this we need to show that the following two statements hold.
\begin{enumerate}
    \item For every $\hat{v}\in U_r\pars{\mcal{M}\cap B\pars{u,r}-u}$ there exists a $\hat{w}\in U_r\pars{\mbb{T}_u\mcal{M}}$ such that $\norm{\hat{v}-\hat{w}}\le 2Cr^2$.
    \item For every $\hat{w}\in U_r\pars{\mbb{T}_u\mcal{M}}$ there exists a $\hat{v}\in U_r\pars{\mcal{M}\cap B\pars{u,r}-u}$ such that $\norm{\hat{v}-\hat{w}}\le 2Cr^2$.
\end{enumerate}

\paragraph{Proof of 1.}
Let $\hat{v}\in U_r\pars{\mcal{M}\cap B\pars{u,r}-u}$ and let $v\in\mcal{M}\cap B\pars{u,r}-u$ be any element that satisfies $U_r\pars{\braces{v}}=\braces{\hat{v}}$.
In the proof of Theorem~\ref{prop:nhood_properties:tangent_space_projection} we have shown that there exists a $w\in\mbb{T}_u\mcal{M}$ that satisfies $\norm{v-w}\le C\norm{v}^2$ (cf.~Equation~\eqref{eq:v-w_bound}).
We use this $w$ to define 
\begin{equation}
    \tilde{v} := \frac{r}{\norm{v}}v,
    \quad
    \tilde{w} := \frac{r}{\norm{v}}w,
    \quad\text{and}\quad
    \hat{w} = \frac{r}{\norm{w}}w
\end{equation}
and observe that $\tilde{v} = \hat{v} \in U_r\pars{\mbb{T}_u\mcal{M}}$ and that $\hat{w}\in U_r\pars{\mbb{T}_u\mcal{M}}$.
Moreover, $\norm{\hat{v}-\hat{w}} \le \norm{\hat{v}-\tilde{w}} + \norm{\tilde{w}-\hat{w}}$ and $\norm{\tilde{v}-\tilde{w}} = \frac{r}{\norm{v}}\norm{v-w} \le Cr\norm{v} \le Cr^2$.
It thus remains to show that $\norm{\tilde{w}-\hat{w}} \le Cr^2$.

To see this we consider the intersection of $\mcal{M}-u$ and $\mbb{T}_u\mcal{M}$ with the plane $\inner{0,v,w}$.
This is illustrated in Figure~\ref{fig:unitization}.
Since all the points that we have defined so far reside in this plane, the distances between them are preserved and we can henceforth consider only this two-dimensional problem.

To show $a := \norm{\tilde{w}-\hat{w}} \le \norm{\tilde{w}-\tilde{v}} =: b$, we consider the triangle $\Delta\pars{\tilde{v},\tilde{w},0}$ and employ the Pythagorean theorem
\begin{equation}
    r^2 = \pars{r-a}^2 + b^2 .
\end{equation}
Expanding the product and rearranging the terms results in the equation $b^2 = 2ra - a^2$.
Since $r\ge a$ also $2ra \ge 2a^2$.
Therefore, $b^2 \ge 2a^2 - a^2 = a^2$ which is what we wanted to prove.

\paragraph{Proof of 2.}
Let $\hat{w}\in U_r\pars{\mbb{T}_u\mcal{M}}$.
Since $r\le R$, Theorem~\ref{prop:nhood_properties:nhood_projection} guarantees that there exists a $v\in\mcal{M}\cap B\pars{u,r}-u$ such that $\norm{\hat{w}-v}\le Cr^2$.
Let $\hat{v} := \frac{r}{\norm{v}}v$ and observe that, by the reverse triangle inequality,
\begin{equation}
    \norm{\hat{w}} - \norm{v} \le \abs{\norm{\hat{w}} - \norm{v}} \le \norm{\hat{w}-v} \le Cr^2 .
\end{equation}
Rearranging the terms and substituting $\norm{\hat{w}} = r$ then yields $\norm{v} \ge r-Cr^2$.
It is now easy to estimate
\begin{equation}
    \norm{v-\hat{v}} = \abs*{1-\frac{r}{\norm{v}}}\norm{v} 
    = r-\norm{v} \le Cr^2 .
\end{equation}
Finally, using the triangle inequality, we obtain $\norm{\hat{w}-\hat{v}}\le \norm{\hat{w}-v} + \norm{v-\hat{v}}\le 2Cr^2$.
This concludes the proof.

\begin{figure}[ht]
    \centering
    \newcommand{\degre}{\ensuremath{^\circ}}
    \definecolor{angleColor}{rgb}{0.5,0.5,0.5}
    \definecolor{denim}{rgb}{0.08, 0.38, 0.74}
    \tikzset{
    dot/.style = {circle, fill, minimum size=#1, inner sep=0pt, outer sep=0pt},
    dot/.default = 3.5pt  
    }
    \newcommand*\xlim{7.5}
    \newcommand*\curvature{0.2}
    \newcommand*\radius{5.5}
    \newcommand*\vx{4.0}
    \pgfmathdeclarefunction{Parabola}{1}{\pgfmathparse{-(#1)^2*\curvature/2}}
    \begin{tikzpicture}[line cap=round,line join=round, scale=0.865]
    
    \clip(-\xlim,-4) rectangle (\xlim,1);  
    
    \coordinate (u) at (0,0);
	\coordinate (v) at ($(\vx,{Parabola(\vx)})$);
	\path let \p1 = (v) in  coordinate (w) at (\x1,0);
    \coordinate (hat_w) at (\radius,0);
    
    \draw [line width=1.2pt,domain=-6.5:6.5,samples=50,name path=parabola] plot (\x, {Parabola(\x)});
    {
        \def\Mx{-5.75}
        \draw[color=black] ($(\Mx,{Parabola(\Mx)})$) node[anchor=south east] {$\mcal{M}-u$};
    }
    \draw [line width=1.2pt,domain=-\xlim:\xlim,name path=line] plot (\x, {0*\x});
    {
        \def\TuMx{-6.5}
        \draw[color=black] (\TuMx,0) node[anchor=south] {$\mbb{T}_u\mcal{M}$};
    }
    \draw [name path=circle] (u) circle (\radius);
    
	\path [name path=ray] (u) -- ($(u)!2*\radius!(v)$);  
    \path [name intersections={of = ray and circle}];
	\coordinate (hat_v) at (intersection-1);
    \path let \p1 = (hat_v) in coordinate (tilde_w) at (\x1,0);
    
    \node[above] at ({-\radius/2},0) {$r$};
    \node[dot, label=above:$0$] at (u) {};
    \node[dot, label=below:$v$] at (v) {};
    \node[dot, label=above:$w$] at (w) {};
    \node[dot, label=above right:$\hat{w}$] at (hat_w) {};
    \node[dot, label={right:$\tilde{v}=\hat{v}$}] at (hat_v) {};
    \node[dot, label={above:$\tilde{w}$}] at (tilde_w) {};
    
    \draw (v) -- (w);
    \draw [dash pattern=on 3pt off 3pt] (u) -- (hat_v) -- (tilde_w);
    
    \end{tikzpicture}
    \caption{}
    \label{fig:unitization}
\end{figure}

\end{document}